\documentclass[reqno]{amsart}
\usepackage{amsmath}
\usepackage{amsthm, amscd, amssymb, amsfonts, amsbsy}
\usepackage[usenames, dvipsnames]{color}
\usepackage{enumerate}
\usepackage{hyperref}
\usepackage{verbatim}

\numberwithin{equation}{section}

\theoremstyle{plain}
\newtheorem{theorem}{Theorem}[section]
\newtheorem{lemma}[theorem]{Lemma}

\newtheorem{proposition}[theorem]{Proposition}

\theoremstyle{definition}
\newtheorem{assumption}{Assumption}[section]

\theoremstyle{remark}
\newtheorem{remark}{Remark}[section]

\makeatletter
\def\dashint{\operatorname%
{\,\,\text{\bf--}\kern-.98em\DOTSI\intop\ilimits@\!\!}}
\makeatother

\def\bR{\mathbb{R}}
\def\cB{\mathcal{B}}
\def\cE{\mathcal{E}}
\def\cL{\mathcal{L}}
\def\cM{\mathcal{M}}

\begin{document}
\title[Conormal derivative problems]{Conormal derivative problems for stationary Stokes system in Sobolev spaces}

\author[J. Choi]{Jongkeun Choi}
\address[J. Choi]{Department of Mathematics, Korea University, 145 Anam-ro, Seongbuk-gu, Seoul, 02841, Republic of Korea}
\email{jongkeun\_choi@korea.ac.kr}

\thanks{J. Choi was supported by a Korea University Grant}

\author[H. Dong]{Hongjie Dong}
\address[H. Dong]{Division of Applied Mathematics, Brown University, 182 George Street, Providence, RI 02912, USA}

\email{Hongjie\_Dong@brown.edu}

\thanks{H. Dong was partially supported by the NSF under agreement DMS-1600593.}

\author[D. Kim]{Doyoon Kim}
\address[D. Kim]{Department of Mathematics, Korea University, 145 Anam-ro, Seongbuk-gu, Seoul, 02841, Republic of Korea}

\email{doyoon\_kim@korea.ac.kr}

\thanks{D. Kim was supported by Basic Science Research Program through the National Research Foundation of Korea (NRF) funded by the Ministry of Education (2014R1A1A2054865).}

%%\date{Received: date / Revised version: date}

\begin{abstract}
We prove the solvability in Sobolev spaces of the conormal derivative problem for the stationary Stokes system with irregular coefficients on bounded Reifenberg flat domains.
The coefficients are assumed to be merely measurable in one direction, which may differ depending on the local coordinate systems, and have small mean oscillations in the other directions.
In the course of the proof, we use a local version of the Poincar\'e inequality on Reifenberg flat domains,
the proof of which is of independent interest.
\end{abstract}

\maketitle

\section{Introduction}

We study $L_q$ theory of the conormal derivative problem for the stationary Stokes system with variable coefficients:
\begin{equation}		\label{170622@eq1}
\left\{
\begin{aligned}
\operatorname{div} u&=g \quad &\text{in }\, \Omega,\\
\cL u+\nabla p &=f+D_\alpha f^\alpha \quad &\text{in }\, \Omega,\\
\cB u+n p&=f^\alpha n_\alpha \quad &\text{on }\, \partial \Omega,
\end{aligned}
\right.
\end{equation}
where $\Omega$ is a bounded domain in $\bR^d$ and $n=(n_1,\ldots, n_d)^T$ is the outward unit normal to $\partial \Omega$.
The differential operator $\cL$ is in divergence form acting on column vector valued functions $u=(u_1,\ldots,u_d)^T$ as follows:
$$
\cL u=D_\alpha(A^{\alpha\beta}D_\beta u).
$$
We denote by $\cB u=A^{\alpha\beta}D_\beta u n_\alpha$ the conormal derivative of $u$ on the boundary of $\Omega$ associated with the operator $\cL$.

Throughout the paper, the coefficients $ A^{\alpha\beta}= A^{\alpha\beta}(x)$ are $d\times d$ matrix valued functions on $\bR^d$ with the entries $A_{ij}^{\alpha\beta}$ satisfying the strong ellipticity condition; see \eqref{160802@eq6a}.
We assume that the coefficients $A^{\alpha\beta}$ are merely measurable in one direction and have small mean oscillations in the other directions (partially BMO).
For more a precise definition of partially BMO coefficients, see Assumption \ref{170219@ass2}.
Stokes systems with this type of variable coefficients may be used to describe the motion of inhomogeneous fluids with density dependent viscosity and two fluids with interfacial boundaries; see \cite{arXiv:1604.02690v2, arXiv:1702.07045v1} and the references therein.
We note that Stokes systems with variable coefficients can also occur when performing a change of coordinates or when flattening the boundary. See \cite{arXiv:1705.02736}.

Extensive literature exists regarding the regularity theory for Stokes systems.
With respect to the classical Stokes system
$$
\Delta u+\nabla p=f
$$
with Dirichlet and Neumann boundary conditions,
we refer the reader to Fabes--Kenig--Verchota \cite{MR975121}, Kozlov--Maz'ya--Rossmann \cite{MR1788991}, and Boyer--Fabrie \cite{MR2986590}.
In \cite{MR975121}, the authors studied Dirichlet and Neumann boundary value problems on arbitrary Lipschitz domains with $q$ being in a restricted range.
For this line of research, see \cite{MR975122, MR2839049, MR2987056} and the references therein.
In \cite{MR1788991}, the authors considered the system on polyhedral domains.
Later, Maz'ya--Rossmann \cite{MR2228352, MR2321139} treated the Stokes system with a mixed boundary condition (containing the Neumann boundary condition) on polyhedral domains.
The authors in \cite{MR2986590} proved $L_2$-estimates for derivatives of solutions to the Stokes system on regular domains.
Regarding resolvent estimates for Neumann boundary value problems, we refer to \cite{MR1972873, MR2206829}.
See also \cite{MR641818} for various regularity results for both linear and nonlinear Stokes systems with regular coefficients subject to Dirichlet and Neumann boundary conditions.

Recently, in \cite{arXiv:1503.07290v3, arXiv:1604.02690v2, arXiv:1702.07045v1} the Dirichlet problem for the stationary Stokes system with irregular coefficients was studied.
In \cite{arXiv:1503.07290v3}, the unique solvability of the problem in Sobolev spaces was proved on a Lipschitz domain with a small Lipschitz constant when the coefficients have vanishing mean oscillations (VMO) with respect to all the variables.
This result was extended by Dong-Kim \cite{arXiv:1604.02690v2} to the case having partially BMO coefficients.
The authors also established in \cite{arXiv:1604.02690v2} a priori $L_q$-estimates on the whole Euclidean space and a half space under the assumption that $A^{\alpha\beta}$ are functions of only one variable with no regularity assumptions. Later, they further generalized their results to the framework of Sobolev spaces with Muckenhoupt weights; see \cite{arXiv:1702.07045v1}.
In particular, they proved the solvability and weighted estimates (with mixed-norm) for the system on a bounded Reifenberg flat domain.
For other results on weighted estimates for Stokes systems, we refer the reader to \cite{MR3614614}, where the authors considered BMO coefficients with small BMO semi-norms and Muckenhoupt weights in $A_{q/2}$, $q \in (2,\infty)$.

In this paper, we derive analogous results to those in \cite{arXiv:1604.02690v2} when the system has conormal derivative boundary conditions instead of Dirichlet boundary conditions.
More precisely, we prove the solvability in Sobolev spaces (without weights) and the $L_q$-estimate for the conormal derivative problem \eqref{170622@eq1} with partially BMO coefficients in a bounded Reifenberg flat domain $\Omega$ (see Theorem \ref{MT1}).
In particular, for the uniqueness of solutions $u$ to \eqref{170622@eq1}, we impose the normalization condition
$$
\int_\Omega u\,dx=0.
$$
As in \cite{arXiv:1702.07045v1}, using our result in this paper for solutions in Sobolev spaces without weights, one can investigate Stokes systems with conormal derivative boundary conditions in Sobolev spaces with Muckenhoupt weights.

Although, at the conceptual level, the paper is similar to \cite{arXiv:1604.02690v2, arXiv:1702.07045v1}, the technical details are different owing to the distinct nature of the conormal derivative boundary condition (or Neumann type boundary condition).
For instance, in the study of the Dirichlet problem ($u\equiv 0$ on $\partial \Omega$) on a Reifenberg flat domain $\Omega$, the following boundary Sobolev-Poincar\'e inequality
$$
\|u\|_{L_{dq/(d-q)}(\Omega\cap B_r(x_0))}\le N\|Du\|_{L_q(\Omega\cap B_r(x_0))}, \quad 1\le q<d,
$$
is available, where the constant $N$ depends only on $d$, $q$, and the flatness of $\partial \Omega$.
This result is an easy consequence of a Poincar\'e type inequality on the ball $B_r(x_0)$ because $u$ can be extended to a function on $B_r(x_0)$ by setting $u\equiv 0$ on $B_r(x_0)\setminus \Omega$.
For the conormal derivative problem,
one may consider a boundary Sobolev-Poincar\'e inequality
\begin{equation}		\label{170702@eq1}
\|u-c\|_{L_{dq/(d-q)}(\Omega\cap B_r(x_0))}\le N\|Du\|_{L_q(\Omega\cap B_R(x_0))}, \quad 1\le q<d,
\end{equation}
where
$$
c=\dashint_{\Omega\cap B_r(x_0)} u\,dx.
$$
It is well known that the inequality \eqref{170702@eq1} holds if $\Omega \cap B_r(x_0)$ and $\Omega \cap B_R(x_0)$ are replaced by a Lipschtiz domain $\Omega$ or a Reifenberg flat domain $\Omega$ since these domains are in the category of extension domains.
It is also known that if $\Omega$ is a Lipschitz domain, the above inequality holds with a constant $N$ depending only on $d$, $q$, and the Lipschitz constant of $ \Omega$; see, for instance, \cite[Lemma 8.1]{MR3025530}.
However, if $\Omega$ is a Reifenberg flat domain, it is not quite obvious that the inequality \eqref{170702@eq1} follows from the same type of inequality for $\Omega$ because the intersection may not retain the same nice properties as those of $\Omega$.
We were unable to find any literature dealing with the inequality \eqref{170702@eq1} on a Reifenberg flat domain $\Omega$ intersected with a ball.
To show the exact information on the parameters that the constant $N$ depends on, we provide a proof of \eqref{170702@eq1} in Appendix.
On the other hand, in the proof of \cite[Corollary 3]{MR2107043}, Byun-Wang used such type of inequality without a proof, referring to the Sobolev inequality on extension domains.
See also \cite{MR2139880, MR2373207}, in which conormal derivative problems for parabolic equations are considered.
The inequality \eqref{170702@eq1} is the key ingredient in establishing reverse H\"older's inequality of conormal derivative problems for the Stokes system; see Section \ref{170703@sec1}.

In a subsequent paper, we will study Green functions for the Stokes system with the conormal derivative boundary condition.
We note that $L_q$-estimates for boundary value problems play an essential role in the study of Green functions.
For instance, in \cite{MR3105752}, the authors obtained global pointwise estimates of Green functions for elliptic systems with conormal boundary conditions by using $L_q$-estimates for the system.

The remainder of this paper is organized as follows.
In Section \ref{S2}, we state our main result along with some notation and assumptions.
In Section \ref{S3}, we provide some auxiliary results, and in Section \ref{S4}, we establish interior and boundary Lipschitz estimates for solutions.
Finally, in Section \ref{S5} we prove the main theorem using a level set argument.
In Appendix, we provide the proof of a local version of the Poincar\'e inequality on a Reifenberg flat domain.

\section{Main results}		\label{S2}

Throughout this paper, we denote by $\Omega$ a domain in the Euclidean space $\bR^d$, where $d\ge 2$.
For any $x\in \Omega$ and $r>0$, we write $\Omega_r(x)=\Omega \cap B_r(x)$, where $B_r(x)$ is a usual Euclidean ball of radius $r$ centered at $x$.
We also denote $B^+_r(x)=B_r(x)\cap \bR^d_+$, where
$$
\bR^d_+=\{x=(x_1,x')\in \bR^d:x_1>0, \ x'\in \bR^{d-1}\}.
$$
A ball in $\bR^{d-1}$ is denoted by
$$
B_r'(x')=\{y'\in \bR^{d-1}:|x'-y'|<r\}.
$$
We use the abbreviations $B_r:=B_r(0)$ and $B_r^+:=B_r^+(0)$, where $0\in \bR^d$, and $B_r':=B_r'(0)$, where $0\in \bR^{d-1}$.

For $1\le q\le \infty$, we denote by $W^1_q(\Omega)$ the usual Sobolev space and by $\mathring{W}^1_q(\Omega)$ the completion of $C^\infty_0(\Omega)$ in $W^1_q(\Omega)$.
When $|\Omega|<\infty$, we define
$$
\tilde{L}_q(\Omega)=\{u\in L_q(\Omega): (u)_\Omega=0\}, \quad \tilde{W}^1_{q}(\Omega)=\{u\in W^1_{q}(\Omega): (u)_{\Omega}=0\},
$$
where $(u)_\Omega$ is the average of $u$ over $\Omega$, i.e.,
$$
(u)_\Omega=\dashint_\Omega u\,dx=\frac{1}{|\Omega|}\int_\Omega u\,dx.
$$

Let $\cL$ be a strongly elliptic operator of the form
$$
\cL u=D_\alpha (A^{\alpha\beta}D_\beta u).
$$
The coefficients $ A^{\alpha\beta}= A^{\alpha\beta}(x)$ are $d\times d$ matrix valued functions on $\bR^d$ with the entries $A_{ij}^{\alpha\beta}$ satisfying the strong ellipticity condition, i.e., there is a constant $\delta\in (0,1]$ such that
\begin{equation}		\label{160802@eq6a}
|A^{\alpha\beta}|\le \delta^{-1}, \quad \sum_{\alpha,\beta=1}^dA^{\alpha\beta}\xi_\beta\cdot \xi_\alpha\ge \delta\sum_{\alpha=1}^d|\xi_\alpha|^2
\end{equation}
for any $x\in \bR^d$ and $\xi_\alpha\in \bR^d$, $\alpha\in \{1,\ldots,d\}$.
We denote by $\cB u=A^{\alpha\beta}D_\beta u n_\alpha$ the conormal derivative of $u$ on the boundary of $\Omega$ associated with the elliptic operator $\cL$.
The $i$-th component of $\cB u$ is given by
$$
(\cB u)_i=A^{\alpha\beta}_{ij}D_\beta u_jn_\alpha,
$$
where $ n=(n_1,\ldots, n_d)^T$ is the outward unit normal to $\partial \Omega$.
Let $q, q_1 \in (1,\infty)$, $q_1\ge qd/(q+d)$, and $\Omega$ be a bounded domain in $\bR^d$.
For $f\in\tilde{L}_{q_1}(\Omega)^d$ and $f^\alpha \in L_{q}(\Omega)^d$, we say that $(u, p)\in W^1_{q}(\Omega)^d\times L_{q}(\Omega)$ is a weak solution of the problem
$$
\left\{
\begin{aligned}
\cL u+\nabla p= f+D_\alpha  f^\alpha \quad &\text{in }\ \Omega,\\
\cB u+np= f^\alpha n_\alpha \quad &\text{on }\ \partial \Omega,
\end{aligned}
\right.
$$
if
\[
\int_{\Omega} A^{\alpha\beta}D_\beta u\cdot D_\alpha \phi\ dx+\int_\Omega p \operatorname{div} \phi\ dx=-\int_{\Omega} f\cdot \phi\ dx+\int_\Omega f^\alpha\cdot D_\alpha \phi \ dx
\]
holds for any $\phi\in W^1_{q/(q-1)}(\Omega)^d$.

\begin{assumption}[$\gamma$]		\label{170219@ass2}
There exists $R_0\in (0,1]$ such that the following hold.
\begin{enumerate}[$(i)$]
\item
For $x_0\in \Omega$ and $0<R\le \min\{R_0,\operatorname{dist}(x_0,\partial \Omega)\}$, there exists a coordinate system depending on $x_0$ and $R$ such that in this new coordinate system, we have that
\begin{equation}		\label{170219@eq1}
\dashint_{B_R(x_0)}\bigg|A^{\alpha\beta}(x_1,y')-\dashint_{B'_R(x'_0)}A^{\alpha\beta}(x_1,y')\,dy'\bigg|\,dx\le \gamma
\end{equation}
\item
For any $x_0\in \partial \Omega$ and $0<R\le R_0$, there is a coordinate system depending on $x_0$ and $R$ such that in the new coordinate system we have that \eqref{170219@eq1} holds, and
$$
\{y:x_{01}+\gamma R<y_1\}\cap B_R(x_0)\subset \Omega_R(x_0)\subset \{y:x_{01}-\gamma R<y_1\}\cap B_R(x_0),
$$
where $x_{01}$ is the first coordinate of $x_0$ in the new coordinate system.
\end{enumerate}
\end{assumption}

The main result of the paper reads as follows.

\begin{theorem}		\label{MT1}
Let $q,q_1\in (1,\infty)$ satisfying $q_1\ge qd/(q+d)$, and let $\Omega$ be a bounded domain with $\operatorname{diam}\Omega\le K$.
Then there exists a constant $\gamma=\gamma(d,\delta,q)\in (0,1/48]$ such that, under Assumption \ref{170219@ass2} $(\gamma)$, the following holds:
for $f\in \tilde{L}_{q_1}(\Omega)^d$, $f^\alpha\in L_q(\Omega)^d$, and $g\in L_q(\Omega)$, there exists a unique $(u,p)\in \tilde{W}^1_q(\Omega)^d\times L_q(\Omega)$ satisfying
\begin{equation}		\label{170512@eq4}
\left\{
\begin{aligned}
\operatorname{div} u=g \quad &\text{in }\ \Omega,\\
\cL u+\nabla p=f+D_\alpha  f^\alpha \quad &\text{in }\ \Omega,\\
\cB u+n p = f^\alpha n_\alpha \quad &\text{on }\ \partial \Omega
\end{aligned}
\right.
\end{equation}
and
\begin{equation}		\label{170512@eq4a}
\|Du\|_{L_q(\Omega)}+\|p\|_{L_q(\Omega)}\le N\|f\|_{L_{q_1}(\Omega)}+\tilde{N}\big(\|f^\alpha\|_{L_q(\Omega)}+\|g\|_{L_q(\Omega)}\big),
\end{equation}
where $N=N(d,\delta,q,q_1,R_0,K)$ and $\tilde{N}=\tilde{N}(d,\delta,q, R_0,K)$.
\end{theorem}

\section{Auxiliary results}		\label{S3}

In this section, we derive some auxiliary results. We impose no regularity assumptions on the coefficients $A^{\alpha\beta}$ of the operator $\cL$.

\subsection{$W^1_2$-solvability}

The lemma below shows that the divergence equation is solvable in $\tilde{W}^1_q(\Omega)^d$ provided that $\Omega$ is bounded.

\begin{lemma}		\label{170224@lem1}
Let $\Omega$ be a bounded domain in $\bR^d$, $q\in (1,\infty)$, and $g\in L_q(\Omega)$.
Then there exists a function $u\in \tilde{W}^1_q(\Omega)^d$ such that
$$
\operatorname{div}u=g \quad \text{in }\, \Omega, \quad \|Du\|_{L_q(\Omega)}\le N\|g\|_{L_q(\Omega)},
$$
where the constant $N$ depends only on $d$ and $q$.
\end{lemma}

\begin{proof}
Assume that $\Omega\subset B_R$ for some $R\ge 1$.
We denote $\bar{g}=g\chi_\Omega$, where $\chi_\Omega$ is the characteristic function.
By the well-known result on the existence of solutions of the divergence equation in a ball, there exists $v\in \mathring{W}^1_{q}(B_R)^d$ such that
$$
\operatorname{div}v=\bar{g}-(\bar{g})_{B_R} \quad \text{in }\, B_R, \quad \|Dv\|_{L_p(B_R)}\le N(d,q)\|\bar{g}\|_{L_{q}(B_R)}.
$$
We define $u=w-(w)_\Omega$, where
$$
w=v+\frac{(\bar{g})_{B_R}}{d}x.
$$
It then follows that $\operatorname{div}u=g$ in $\Omega$.
Moreover, we get
$$
\|Du\|_{L_q(\Omega)}\le N(d,q)\|g\|_{L_q(\Omega)}
$$
because
$$
|Du|=\left|Dv+\frac{(\bar{g})_{B_R}}{d}I\right|\le |Dv|+|(\bar{g})_{B_R}|,
$$
where $I$ is the $d\times d$ identity matrix.
\end{proof}

\begin{lemma}		\label{170224@lem2}
Let $\Omega$ be a bounded domain in $\bR^d$.
Assume that there exists a constant $K_0>0$ such that
\begin{equation}		\label{170513_eq9}
\|\phi\|_{L_2(\Omega)}\le K_0\|D\phi\|_{L_2(\Omega)} \quad \text{for all }\, \phi\in \tilde{W}^1_2(\Omega).
\end{equation}
Then, for any $ f^\alpha\in L_2(\Omega)^d$ and $g\in L_2(\Omega)$, there exists a unique $(u,p)\in \tilde{W}^1_2(\Omega)^d\times {L}_2(\Omega)$ satisfying
\begin{equation}		\label{0511@eq1}
\left\{
\begin{aligned}
\operatorname{div} u=g \quad &\text{in }\ \Omega,\\
\cL u+\nabla p=D_\alpha  f^\alpha \quad &\text{in }\ \Omega,\\
\cB u+n p = f^\alpha n_\alpha \quad &\text{on }\ \partial \Omega.
\end{aligned}
\right.
\end{equation}
Moreover, we have
\begin{equation}		\label{170703@eq2}
\|Du\|_{L_2(\Omega)}+\|p\|_{L_2(\Omega)}\le N\left(\|f^\alpha\|_{L_2(\Omega)}+\|g\|_{L_2(\Omega)}\right),
\end{equation}
where $N=N(d,\delta)$.
\end{lemma}

\begin{proof}
By \eqref{170513_eq9}, $\tilde{W}^1_2(\Omega)$ can be understood as a Hilbert space with the inner product
$$
\langle u,v\rangle =\int_\Omega D_\alpha u\cdot D_\alpha v\,dx.
$$
The proof of the lemma is then nearly the same as that of \cite[Lemma 3.1]{arXiv:1503.07290v3}, by using Lemma \ref{170224@lem1} instead of the condition (D) in \cite[Section 3.1]{arXiv:1503.07290v3}.
We note that the constant $K_0$ in \eqref{170513_eq9} does not appear in the derivation of the estimate \eqref{170703@eq2}.
We omit the details.
\end{proof}

\begin{remark}
If $\Omega$ is a bounded Reifenberg flat domain as in Assumption \ref{170219@ass2} (ii),
the assumption in Lemma \ref{170224@lem2} is satisfied.
In other words, because $\Omega$ is an extension domain (see, for instance, \cite{MR631089, MR3186805}), the Poincar\'e inequality \eqref{170513_eq9} holds.
We refer the reader to \cite[pp. 286--290]{MR2597943} for more details.
We note that the arguments in \cite[pp. 286--290]{MR2597943} hold on a bounded extension domain.

\end{remark}

\subsection{Reverse H\"older's inequality}		\label{170703@sec1}

This subsection is devoted to a reverse H\"older's inequality for solutions to the Stokes system with the conormal boundary condition.

\begin{assumption}		\label{170226@ass1}
Let $\gamma\in [0,1/48]$.
There exists a positive constant $R_0$ such that the following holds:
for any $x_0\in \partial \Omega$ and $R\in (0,R_0]$, there is a coordinate system depending on $x_0$ and $R$ such that in this new coordinate system (called the coordinate system associated with $(x_0,R)$), we have
\begin{equation}		\label{170512_eq1}
\{y:x_{01}+\gamma R<y_1\}\cap B_R(x_0)\subset \Omega_R(x_0)\subset \{y:x_{01}-\gamma R<y_1\}\cap B_R(x_0),
\end{equation}
where $x_{01}$ is the first coordinate of $x_0$ in the new coordinate system.
\end{assumption}

If $\Omega$ is a bounded Reifenberg flat domain, then the Poincar\'e inequality holds over $\Omega$.
However, the domain of the Poincar\'e inequality presented in the theorem below is $\Omega \cap B_R(x_0)$, $x_0 \in \partial\Omega$, which is not a Reifenberg flat domain with the same flatness as that of $\Omega$.
Moreover, we need correct information on the parameters on which the constant of the Poincar\'e inequality depends.
Thus, for the reader's convenience, we provide a proof of the theorem in Appendix.

\begin{theorem}[Poincar\'e inequality]		\label{170503@thm1}
Let $\Omega\subset \bR^d$ be a Reifenberg flat domain satisfying Assumption \ref{170226@ass1}.
Let $x_0\in \partial\Omega$ and $R\in (0, R_0/4]$.
Then, for any $1<q<d$ and $u\in W^1_q(\Omega)$, we have
$$
\|u-(u)_{\Omega_R(x_0)}\|_{L_{dq/(d-q)}(\Omega_R(x_0))}\le N\|Du\|_{L_q(\Omega_{2R}(x_0))},
$$
where $N=N(d,q)$.
\end{theorem}

Based on the $L_2$-estimate and Poincar\'e inequality in Theorem \ref{170503@thm1}, we obtain the following estimates for $Du$ and $p$.

\begin{lemma}	\label{170225@lem1}
Let $\frac{2d}{d+2}<q<2$, and let $\Omega\subset \bR^d$ be a Reifenberg flat domain satisfying Assumption \ref{170226@ass1}.
Suppose that $(u,p)\in W^1_2(\Omega)^d\times L_2(\Omega)$ satisfies \eqref{0511@eq1} with $f^\alpha\in L_2(\Omega)^d$ and $g\in L_2(\Omega)$.
Then, for $x_0\in \overline{\Omega}$ and $R\in (0, R_0/8]$ satisfying either
$$
 B_{2R}(x_0)\subset \Omega \quad \text{or}\quad x_0\in \partial \Omega,
$$
we have
\begin{equation}		\label{170225@eq1}
\begin{aligned}
&\int_{\Omega_{R}(x_0)}\big(|Du|^2+|p|^2\big)\,dx
\le \theta\int_{\Omega_{4R}(x_0)}|p|^2\,dx\\
&\quad +NR^{d(1-2/q)}\left(\int_{\Omega_{4R}(x_0)}\big(|Du|^2+|p|^2\big)^{q/2}\,dx\right)^{2/q}+N\int_{\Omega_{4R}(x_0)}\big(|f^\alpha|^2+|g|^2\big)\,dx,
\end{aligned}
\end{equation}
where $\theta\in (0,1)$ and $N=N(d,\delta, q,\theta)$.
\end{lemma}

\begin{proof}
We prove only the case $x_0\in \partial \Omega$ because the other case is the same with obvious modifications.
Without loss of generality, we assume that $x_0=0$.
Let $R\in (0,R_0/8]$ and $\eta$ be a smooth function on $\bR^d$ satisfying
$$
0\le \eta\le 1, \quad \eta\equiv 1 \,\text{ on }\, B_{R}, \quad \operatorname{supp}\eta\subset B_{2R}, \quad |\nabla \eta|\le NR^{-1}.
$$
By applying $\eta^2(u-(u)_{\Omega_{2R}})$ as a test function to \eqref{0511@eq1},
and using both H\"older's and Young's inequalities, we obtain for $\theta\in (0,1)$ that
\begin{equation}		\label{170504@eq2}
\begin{aligned}
\int_{\Omega_{2R}}\eta^2|Du|^2\,dx&\le \frac{N}{R^2}\int_{\Omega_{2R}}|u-(u)_{\Omega_{2R}}|^2\,dx\\
&\quad +\theta \int_{\Omega_{2R}}|p|^2\,dx+N\int_{\Omega_{2R}}\big(|f^\alpha|^2+|g|^2\big)\,dx,
\end{aligned}
\end{equation}
where $N=N(d,\delta,\theta)$.

We extend $p$ by zero on $B_{2R}\setminus\Omega$.
From the existence of solutions to the divergence equation in a ball,
there exists $w\in \mathring{W}^1_2(B_{R})^d$ satisfying
\begin{equation}		\label{170504@eq3}
\begin{aligned}
&\operatorname{div}w=p-(p)_{B_{R}}\quad \text{in }\, B_{R},\\
&\|Dw\|_{L_2(B_{R})}\le N(d)\|p-(p)_{B_{R}}\|_{L_2(B_{R})}.
\end{aligned}
\end{equation}
We extend $w$ to be zero on $\Omega \setminus \Omega_R$ and apply $w$ as a test function to \eqref{0511@eq1} to get
$$
\int_{\Omega_{R}} p \operatorname{div}w\,dx=-\int_\Omega A^{\alpha\beta}D_\beta u \cdot D_\alpha w\,dx+\int_\Omega f^\alpha \cdot D_\alpha w\,dx.
$$
Using \eqref{170504@eq3} with the fact that
$$
\int_{\Omega_{R}}p\operatorname{div}w\,dx=\int_{B_{R}}(p-(p)_{B_{R}})\operatorname{div}w\,dx,
$$
we have
$$
\int_{B_{R}}|p-(p)_{B_{R}}|^2\,dx\le N\int_{\Omega_{R}}|Du|^2\,dx+N\int_{\Omega_{R}} |f^\alpha|^2\,dx,
$$
and thus, we get from \eqref{170504@eq2} that
\begin{align*}
\int_{\Omega_{R}}|p|^2\,dx
&\le \frac{N}{R^2}\int_{\Omega_{2R}}|u-(u)_{\Omega_{2R}}|^2\,dx+\theta \int_{\Omega_{2R}}|p|^2\,dx\\
&\quad +NR^{d(1-2/q)}\left(\int_{\Omega_{2R}}|p|^q\,dx\right)^{2/q}+N\int_{\Omega_{2R}}\big(|f^\alpha|^2+|g|^2\big)\,dx
\end{align*}
for any $\theta\in (0,1)$, where $N=N(d,\delta,q,\theta)$.
This together with \eqref{170504@eq2} yields
\begin{equation}		\label{170506@eq4}
\begin{aligned}
&\int_{\Omega_{R}}\big(|Du|^2+|p|^2\big)\,dx
\le \frac{N}{R^2}\int_{\Omega_{2R}}|u-(u)_{\Omega_{2R}}|^2\,dx
+\theta \int_{\Omega_{2R}}|p|^2\,dx\\
&\quad +NR^{d(1-2/q)}\left(\int_{\Omega_{2R}}|p|^q\,dx\right)^{2/q}+N\int_{\Omega_{2R}}\big(|f^\alpha|^2+|g|^2\big)\,dx.
\end{aligned}
\end{equation}
From H\"older's inequality and the Poincar\'e inequality in Theorem \ref{170503@thm1}, it follows that
\begin{align*}
\frac{1}{R^2}\int_{\Omega_{2R}}|u-(u)_{\Omega_{2R}}|^2\,dx&\le NR^{d(1-2/q)}\|u-(u)_{\Omega_{2R}}\|_{L_{\frac{dq}{d-q}}
(\Omega_{2R})}^2\\
&\le NR^{d(1-2/q)}\|Du\|_{L_q(\Omega_{4R})}^{2}.
\end{align*}
Combining \eqref{170506@eq4} and the above inequality, we conclude the desired estimate.
The lemma is proved.
\end{proof}

Using Lemma \ref{170225@lem1} and Gehring's lemma, we get the following reverse H\"older's inequality.

\begin{lemma}		\label{170506@lem1}
Let $q_1>2$ and $\Omega\subset \bR^d$ be a Reifenberg flat domain satisfying Assumption \ref{170226@ass1}.
Suppose that $(u,p)\in W^1_2(\Omega)^d\times L_2(\Omega)$ satisfies \eqref{0511@eq1} with $f^\alpha\in L_{q_1}(\Omega)^d$ and $g\in L_{q_1}(\Omega)$.
Then there exist constants $q_0\in (2,q_1)$ and $N>0$, depending only on $d$, $\delta$, and $q_1$
such that
$$
\big(|D\bar{u}|^{q_0}+|\bar{p}|^{q_0}\big)^{1/q_0}_{B_R(x_0)}\le N\big(|D\bar{u}|^2+|\bar{p}|^2\big)^{1/2}_{B_{2R}(x_0)}+N\big(|\bar{f}^\alpha|^{q_0}+|\bar{g}|^{q_0}\big)^{1/q_0}_{B_{2R}(x_0)}
$$
for any $x_0\in \bR^d$ and $R\in (0, R_0]$, where $D\bar{u}$, $\bar{p}$, $\bar{f}^\alpha$, and $\bar{g}$ are the extensions of $Du$, $p$, $f^\alpha$, and $g$ to $\bR^d$ so that they are zero on $\bR^d\setminus \Omega$.

\end{lemma}

\begin{proof}
We fix a constant $q\in (2d/(d+2),2)$, and set
$$
\Phi=\big(|D\bar{u}|^2+|\bar{p}|^2\big)^{q/2}, \quad \Psi=\big(|\bar{f}^\alpha|^2+|\bar{g}|^2\big)^{q/2}.
$$
Then, by Lemma \ref{170225@lem1}, it follows that
\begin{equation}		\label{170509@eq1}
\begin{aligned}
\int_{B_R(x_0)}\Phi^{2/q}\,dx&\le \theta\int_{B_{14R}(x_0)}\Phi^{2/q}\,dx\\
&\quad +NR^{d(1-2/q)}\left(\int_{B_{14R}(x_0)}\Phi\,dx\right)^{2/q}+N\int_{B_{14R}(x_0)}\Psi^{2/q}\,dx
\end{aligned}
\end{equation}
for any $x_0\in \bR^d$, $R\in (0, R_0/24]$, and $\theta\in (0,1)$, where $N=N(d,\delta,\theta)$.
Indeed, if $B_{2R}(x_0)\subset \Omega$, then \eqref{170509@eq1} follows from Lemma \ref{170225@lem1}.
In the case when $B_{2R}(x_0)\cap \partial \Omega\neq \emptyset$, there exists $y_0\in \partial \Omega$ such that $|x_0-y_0|=\operatorname{dist}(x_0,\partial \Omega)\le 2R$ and
$$
B_{R}(x_0)\subset B_{3R}(y_0)\subset B_{12R}(y_0)\subset B_{14R}(x_0).
$$
Using Lemma \ref{170225@lem1} and the fact that $3R\le R_0/8$, we obtain \eqref{170509@eq1} with $B_{3R}(y_0)$ and $B_{12R}(y_0)$ in place of $B_{R}(x_0)$ and $B_{14R}(x_0)$, respectively.
Hence, we get the inequality \eqref{170509@eq1}.
If $B_{2R}(x_0)\subset \bR^d\setminus \Omega$, by the definition of $D\bar{u}$ and $\bar{p}$, \eqref{170509@eq1} trivially holds.

For $x_0\in \bR^d$ and $R\in (0, R_0]$, using a covering argument and  \eqref{170509@eq1} with $y\in B_R(x_0)$ and $R/24$ in place of $x_0$ and $R$, respectively, and taking a sufficiently small $\theta$, we have
$$
\begin{aligned}
\dashint_{B_R(x_0)}\Phi^{2/q}\,dx&\le \frac{1}{2}\dashint_{B_{2R}(x_0)}\Phi^{2/q}\,dx\\
&\quad +N\left(\dashint_{B_{2R}(x_0)}\Phi\,dx\right)^{2/q}+N\dashint_{B_{2R}(x_0)}\Psi^{2/q}\,dx,
\end{aligned}
$$
where $N=N(d,\delta)$.
Therefore, by Gehring's lemma (see, for instance, \cite[Ch. V]{MR717034}), we get the desired estimate.
The lemma is proved.
\end{proof}

\section{$L_\infty$ and H\"older estimates}		\label{S4}

In this section, we prove $L_\infty$-estimates of $Du$ and $p$.
We set
$$
\cL_0 u=D_\alpha(A^{\alpha\beta}_0D_\beta u),
$$
where $A_0^{\alpha\beta}=A_0^{\alpha\beta}(x_1)$ satisfy \eqref{160802@eq6a}.
We also denote by $\cB_0 u=A_0^{\alpha\beta}D_\beta un_\alpha$ the conormal derivative of $u$ associated with $\cL_0$.

We start with the following boundary estimates.
For the corresponding interior estimates, see \cite[Section 3]{arXiv:1604.02690v2}.

\begin{lemma}		\label{170509@lem5}
Let $R>0$ and $(u,p)\in W^1_2(B_R^+)^d\times L_2(B_R^+)$ satisfy
\begin{equation}		\label{170509@EQ3}
\left\{
\begin{aligned}
\operatorname{div}  u=0 \quad &\text{in }\ B_R^+,\\
\cL u+\nabla p= 0 \quad &\text{in }\ B_R^+,\\
\cB u+np= 0 \quad &\text{on }\ B_R\cap \partial \bR^d_+.
\end{aligned}
\right.
\end{equation}
Then we have
\begin{equation}		\label{170509@EQ2}
\int_{B_R^+}|p|^2\,dx\le N\int_{B_R^+}|Du|^2\,dx,
\end{equation}
where $N=N(d,\delta)$.
Moreover, for any $0<r<R$, we have
\begin{equation}		\label{170509@EQ1}
\int_{B_r^+}|Du|^2\,dx\le \frac{N}{(R-r)^2}\int_{B_R^+}|u|^2\,dx,
\end{equation}
where $N=N(d,\delta)$.
\end{lemma}

\begin{remark}
In the above lemma and throughout the paper, $(u,p)\in W^1_2(B^+_R)^d\times L_2(B_R^+)$ is said to satisfy \eqref{170509@EQ3} if
$$
\int_{B_R^+}A^{\alpha\beta}D_\beta u\cdot D_\alpha \phi\,dx+\int_{B_R^+}p \operatorname{div}\phi\,dx=0
$$
for any $\phi\in \mathring{W}^1_2(B_R)^d$.
It is clear that, as a test function, one can also use $\phi \in W_2^1(B_R^+)^d$ such that $\phi = 0$ on $\partial B_R \cap \bR^d_+$.
\end{remark}

\begin{proof}[Proof of Lemma \ref{170509@lem5}]
To show \eqref{170509@EQ2},
we extend $p$ to $B_R$ so that $(p)_{B_R}=0$ and $\|p\|_{L_2(B_R)}$ is comparable to $\|p\|_{L_2(B_R^+)}$.
By the existence of solutions to the divergence equation in a ball, there exists $w\in \mathring{W}^1_2(B_R)^d$ satisfying
$$
\operatorname{div}w= p \, \text{ in }\, B_R,
\quad  \|Dw\|_{L_2(B_R)}\le N(d)\|p\|_{L_2(B_R^+)}.
$$
Applying $w$ as a test function to \eqref{170509@EQ3}, we have
$$
\int_{B_R^+}|p|^2\,dx= -\int_{B_R^+}A^{\alpha\beta}D_\beta u\cdot D_\alpha w\,dx,
$$
and thus, we get  \eqref{170509@EQ2}.
The inequality \eqref{170509@EQ1} is deduced from \eqref{170509@EQ2} in the same way as \cite[Lemma 3.7]{arXiv:1604.02690v2} is deduced from \cite[Lemmas 3.4 and 3.6]{arXiv:1604.02690v2}.
The lemma is proved.
\end{proof}

Using the standard finite difference argument, we obtain the following estimates for $DD_{x'}u$ and $D_{x'}p$.
The corresponding interior estimates can be found in \cite[Lemmas 4.1 and 4.2]{arXiv:1604.02690v2}.

\begin{lemma}		\label{160813@lem1}
Let $0<r<R$
and $(u,p)\in W^1_2(B_R^+)^d\times L_2(B_R^+)$ satisfies
$$
\left\{
\begin{aligned}
\operatorname{div}  u=0 \quad &\text{in }\ B_R^+,\\
\cL_0 u+\nabla p= 0 \quad &\text{in }\ B_R^+,\\
\cB_0 u+np=0 \quad &\text{on }\ B_R\cap \partial \bR^d_+.
\end{aligned}
\right.
$$
Then we have
$$
\int_{B_r^+}|DD_{x'}u|^2\,dx+\int_{B_r^+}|D_{x'}p|\,dx\le \frac{N}{(R-r)^2}\int_{B_R^+}|Du|^2\,dx,
$$
where $N=N(d,\delta)$.
\end{lemma}

\begin{proof}
Denote by $\delta_{i,h}f$ the $i$-th difference quotient of $f$ with step size $h$, i.e.,
$$
\delta_{i,h}f(x)=\frac{f(x+he_i)-f(x)}{h}.
$$
Let
$$
R_1=\frac{R+r}{2} \quad \text{and} \quad 0<|h|<\frac{R-r}{2}.
$$
Since the coefficients are functions of only $x_1$, we obtain that
$$
\left\{
\begin{aligned}
\operatorname{div} (\delta_{i,h}u)=0 \quad &\text{in }\ B_{R_1}^+,\\
\cL_0 (\delta_{i,h}u)+\nabla (\delta_{i,h}p)= 0 \quad &\text{in }\ B_{R_1}^+,\\
\cB_0 (\delta_{i,h}u)+n(\delta_{i,h}p)=0 \quad &\text{on }\ B_{R_1}\cap \partial \bR^d_+,
\end{aligned}
\right.
$$
where $i=2,\ldots,d$.
Then, by Lemma \ref{170509@lem5}, we have
$$
\int_{B_r^+}|\delta_{i,h}p|^2\,dx\le N\int_{B_r^+}|D(\delta_{i,h}u)|^2\,dx\le \frac{N}{(R-r)^2}\int_{B_{R_1}^+}|\delta_{i,h}u|^2\,dx,
$$
where $N=N(d,\delta)$.
This, along with the standard finite difference argument (see, for instance, \cite[Section 5.8.2]{MR2597943}), implies the desired estimate.
\end{proof}

In the lemma below, we obtain $L_\infty$-estimates for $Du$ and $p$ when $(u,p)$ is a weak solution of $\cL_0u+\nabla p=0$.
We also prove H\"older semi-norm estimates for linear combinations of derivatives of $u$.
Indeed, we do not use H\"older semi-norm estimates in this paper, but we present here the results for later use of the estimates in the study of weighted $L_q$-estimates.
For the Dirichlet counterpart of the estimates and their application to $L_q$-estimates with Muckenhoupt weights, see \cite{arXiv:1702.07045v1}.

As usual, the H\"older semi-norm of $u$ is defined by
$$
[u]_{C^\gamma(\Omega)}=\sup_{\substack{x,y\in \Omega \\ x\neq y}} \frac{|u(x)-u(y)|}{|x-y|^\gamma}.
$$
Using the fact that $A_0^{\alpha\beta}$ are independent of $x' \in \bR^{d-1}$, we observe that
\begin{equation}		\label{170613@eq1}
\cL_0u+\nabla p=D_1 U+\sum_{\alpha=2}^d(A^{\alpha\beta}_0D_{\alpha\beta}u+D_\alpha p e_\alpha),
\end{equation}
if $u$ and $p$ are sufficiently smooth, where $e_\alpha$ is the $\alpha$-th unit vector in $\bR^d$ and
$$
U:=A^{1\beta}_0D_\beta u+p e_1.
$$
In other words, we have
$$
U_1=(A^{1\beta}_0)_{1j}D_\beta u_j+p, \quad U_i=(A^{1\beta}_0)_{ij}D_\beta u_j, \quad i=2,\ldots,d.
$$

\begin{lemma}		\label{160815@lem1}
There exists a constant $N=N(d,\delta)$ such that the following hold.
\begin{enumerate}[$(a)$]
\item
If $(u,p)\in W^1_2(B_2)^d\times L_2(B_2)$ satisfies
\begin{equation*}%		\label{160816_eq2}
\left\{
\begin{aligned}
\operatorname{div} u=0 \quad &\text{in }\ B_2,\\
\cL_0 u+\nabla p= 0 \quad &\text{in }\ B_2,
\end{aligned}
\right.
\end{equation*}
then we have
$$
\|Du\|_{L_\infty(B_1)}+[D_{x'}u]_{C^{1/2}(B_1)}+[U]_{C^{1/2}(B_1)}\le N\|Du\|_{L_2(B_2)},
$$
$$
\|p\|_{L_\infty(B_1)}\le N\|Du\|_{L_2(B_2)}+N\|p\|_{L_2(B_2)}.
$$

\item
If $(u,p)\in W^1_2(B_2^+)^d\times L_2(B_2^+)$ satisfies
\begin{equation*}%		\label{160816_eq2a}
\left\{
\begin{aligned}
\operatorname{div} u=0 \quad &\text{in }\ B_2^+,\\
\cL_0 u+\nabla p= 0 \quad &\text{in }\ B_2^+,\\
\cB_0 u+n p=0 \quad &\text{on }\, B_2\cap \partial \bR^d_+,
\end{aligned}
\right.
\end{equation*}
then we have
$$
\|Du\|_{L_\infty(B_1^+)}+\|p\|_{L_\infty(B_1^+)}+[D_{x'}u]_{C^{1/2}(B_1^+)}+[U]_{C^{1/2}(B_1^+)}\le N\|Du\|_{L_2(B_2^+)}.
$$
\end{enumerate}
\end{lemma}

\begin{proof}
For the proof of the assertion $(a)$, we refer to \cite[Lemma 4.1 $(a)$]{arXiv:1702.07045v1}.
To prove the assertion $(b)$, we let $r_1\in (1,2)$ and $i=2,\ldots,d$.
By Lemma \ref{160813@lem1}, we have
$$
(D_{i}u,D_{i}p)\in W^1_2(B^+_{r_1})^d\times L_2(B_{r_1}^+)
$$
and
$$
\int_{B_{r_1}^+}|DD_{i}u|^2\,dx+\int_{B_{r_1}^+}|D_{i}p|^2\,dx\le N\int_{B_2^+}|Du|^2\,dx,
$$
where $N=N(d,\delta,r_1)$.
Moreover, $(D_i u,D_ip)$ satisfies
$$
\left\{
\begin{aligned}
\operatorname{div} (D_iu)=0 \quad &\text{in }\ B_{r_1}^+,\\
\cL_0 (D_i u)+\nabla (D_ip)= 0 \quad &\text{in }\ B_{r_1}^+,\\
\cB_0 (D_iu)+n (D_ip)=0 \quad &\text{on }\, B_{r_1}\cap \partial \bR^d_+.
\end{aligned}
\right.
$$
We then use Lemma \ref{160813@lem1} again as above with $r_2$ in place of $r_1$ and with $r_1$ in place of $2$, where $1<r_2<r_1$.
By repeating this process, we obtain  that
$$
(D_{i}^ku,D_i^kp)\in W^1_2(B_r^+)^d\times L_2(B_r^+)
$$
and
$$
\|DD_{i}^ku\|_{L_2(B_r^+)}+\|D_{i}^kp\|_{L_2(B_r^+)}\le N\|Du\|_{L_2(B_2^+)}
$$
for any $r\in [1,2)$ and $k\in \{1,2,\ldots\}$, where $N=N(d,\delta,r,k)$.
Since the above inequality holds for $i=2,\ldots,d$, we have
\begin{equation}		\label{160815@eq8a}
\|DD_{x'}^ku\|_{L_2(B_r^+)}+\|D_{x'}^kp\|_{L_2(B_r^+)}\le N\|Du\|_{L_2(B_2^+)}
\end{equation}
for any $r\in [1,2)$ and $k\in \{0,1,2,\ldots\}$, where we used Lemma \ref{170509@lem5} for the case when $k=0$.
Then, using \eqref{160815@eq8a} and an anisotropic Sobolev embedding theorem with $k>(d-1)/2$ (see, for instance, the proof of \cite[Lemma 3.5]{MR2800569}), we get
$$
\|D_{x'}u\|_{L_\infty(B_1^+)}+[D_{x'}u]_{C^{1/2}(B_1^+)}\le N\|Du\|_{L_2(B_2^+)}.
$$
Using the relation $\operatorname{div} u=0$, we get from the above inequality that
\begin{equation}		\label{170612@eq1}
\|D_1u_1\|_{L_\infty(B_1^+)}+\|D_{x'}u\|_{L_\infty(B_1^+)}+[D_{x'}u]_{C^{1/2}(B_1^+)}\le N\|Du\|_{L_2(B_2^+)}.
\end{equation}

Now we are ready to prove the assertion $(b)$.
From the definition of $U$ and \eqref{160815@eq8a}, it follows that
\begin{align}
\nonumber
\|D_{x'}^k U\|_{L_2(B_1^+)}&\le N\|DD_{x'}^k u\|_{L_2(B_1^+)}+N\|D_{x'}^kp\|_{L_2(B_1^+)}\\
\label{170613@eq1a}
&\le N\|Du\|_{L_2(B_2^+)}
\end{align}
for any $k\in \{0,1,2,\ldots\}$, where $N=N(d,\delta,k)$.
Since $\cL_0u+\nabla p=0$, we obtain by \eqref{170613@eq1} that
$$
D_1U=-\sum_{\alpha=2}^d(A_0^{\alpha\beta}D_{\alpha\beta}u+D_\alpha p e_\alpha).
$$
This together with \eqref{160815@eq8a} yields that $D_1 U$ has sufficiently many derivatives in $x'$ with the estimates
\begin{equation}		\label{170613@eq2}
\|D_1 D_{x'}^k U\|_{L_2(B_1^+)}\le N\|Du\|_{L_2(B_2^+)}
\end{equation}
for any $k\in \{0,1,2,\ldots\}$.
Combining \eqref{170613@eq1a} and \eqref{170613@eq2}, and using the anisotropic Sobolev embedding as above with $k>(d-1)/2$, we have
\begin{equation}		\label{170613@eq3}
\|U\|_{L_\infty(B_1^+)}+[U]_{C^{1/2}(B_1^+)}\le N\|Du\|_{L_2(B_2^+)}.
\end{equation}
Notice from the definition of $U$ that
$$
\sum_{j=2}^d(A^{11}_0)_{ij}D_1u_j=U_i-\sum_{j=1}^d\sum_{\beta=2}^d(A^{1\beta}_0)_{ij}D_\beta u_j-(A^{11}_0)_{i1}D_1u_1, \quad i=2,\ldots,d.
$$
By the ellipticity condition \eqref{160802@eq6a}, $\{A_0\}_{i,j=2}^d$ is nondegenerate, which implies that
$$
\sum_{j=2}^d|D_1u_j(x)|\le N\sum_{i=2}^d |U_i(x)|+N|D_{x'}u(x)|+N|D_1u_1(x)|
$$
for all $x\in B_1^+$, where $N=N(d,\delta)$.
Taking $\|\cdot \|_{L_\infty(B_1^+)}$ of both sides of the above inequality and using \eqref{170612@eq1} and \eqref{170613@eq3}, we conclude that
$$
\|Du\|_{L_\infty(B_1^+)}\le N\|Du\|_{L_2(B_2^+)}.
$$
From this, \eqref{170613@eq3}, and the fact that
$$
p =U_1-(A^{1\beta}_0)_{1j}D_\beta u_j,
$$
we get
$$
\|p\|_{L_\infty(B_1^+)}\le N\|Du\|_{L_2(B_2^+)}.
$$
The lemma is proved.
\end{proof}

\begin{remark}
Following the proof of \cite[Lemma 4.2]{arXiv:1702.07045v1}, we can easily observe that the estimates in Lemma \ref{160815@lem1} still hold under the assumption that $(u,p) \in W_q^1(B_2)^d \times L_q(B_2)$ or $(u,p) \in W_q^1(B_2^+)^d \times L_q(B_2^+)$, where $q \in (1,\infty)$.
Moreover, the $L_2$ norms on the right-hand side of the estimates can be replaced by the corresponding $L_1$ norms.
\end{remark}

\section{$L_q$-estimates for Stokes system}		\label{S5}

\begin{proposition}		\label{170509_prop1}
Suppose that Assumption \ref{170219@ass2} $(\gamma)$ holds with $\gamma\in (0,1/48]$.
Let $q\in (2,\infty)$ and $(u, p)\in W^1_2(\Omega)^d\times L_2(\Omega)$ satisfy
\begin{equation}		\label{170511@eq2}
\left\{
\begin{aligned}
\operatorname{div} u=g &\quad \text{in }\, \Omega,\\
\cL u+\nabla p=D_\alpha f^\alpha &\quad \text{in }\, \Omega,\\
\cB u+np=f^\alpha n_\alpha &\quad \text{on }\, \partial \Omega,
\end{aligned}
\right.
\end{equation}
where $f^\alpha\in L_q(\Omega)^d$ and $g\in L_q(\Omega)$.
Then, for $x_0\in \overline{\Omega}$ and $R\in (0,R_0/2]$ satisfying either
$$
B_R(x_0)\subset \Omega \quad \text{or}\quad x_0\in \partial \Omega,
$$
there exist
$$
(W,p_1), \, (V,p_2)\in L_2(\Omega_R(x_0))^{d\times d}\times L_2(\Omega_R(x_0))
$$
such that $(Du,p)=(W,p_1)+(V,p_2)$ in $\Omega_R(x_0)$ and
\begin{align}
\nonumber
&(|W|^2+|p_1|^2)^{\frac{1}{2}}_{\Omega_R(x_0)}\\
\label{170510@eq1b}
&\quad \le N\gamma^{\frac{1}{2\nu}}(|Du|^2+|p|^2)^{\frac{1}{2}}_{\Omega_{2R}(x_0)}+N(|f^\alpha|^{2\mu}+|g|^{2\mu})^{\frac{1}{2\mu}}_{\Omega_{2R}(x_0)},\\
\nonumber
&\|V\|_{L_\infty(\Omega_{R/4}(x_0))}+\|p_2\|_{L_\infty(\Omega_{R/4}(x_0))}\\
\label{170510@eq1c}
&\quad \le N\big(\gamma^{\frac{1}{2\nu}}+1\big)(|Du|^2+|p|^2)^{\frac{1}{2}}_{\Omega_{2R}(x_0)}+N(|f^\alpha|^{2\mu}+|g|^{2\mu})^{\frac{1}{2\mu}}_{\Omega_{2R}(x_0)},
\end{align}
where $\mu,\, \nu >1$, $2\mu<q$, and $1/\mu+1/\nu=1$.
Here, the constants $\mu$, $\nu$, and $N>0$ depend only on $d$, $\delta$, and $q$.
\end{proposition}

\begin{remark}
Proposition \ref{170509_prop1} is a version of \cite[Proposition 5.1]{arXiv:1702.07045v1} when the boundary condition is of the Neumann type.
Thus, the proof of Proposition \ref{170509_prop1} is an adaptation of that of \cite[Proposition 5.1]{arXiv:1702.07045v1}.
Since the estimates for the case with $B_R(x_0) \subset \Omega$ (i.e., the interior estimates) are irrelevant to the boundary condition, we may refer to the corresponding case in \cite[Proposition 5.1]{arXiv:1702.07045v1}.
However, we provide here the proof of both the boundary and interior cases showing that the constants $N$ in the estimates \eqref{170510@eq1b} and \eqref{170510@eq1c} depend only on $d$, $\delta$, and $q$.
On the other hand, the interior estimates in \cite[Proposition 5.1]{arXiv:1702.07045v1} have constants $N$ depending also on other parameters such as $R_0$ in Assumption \ref{170219@ass2} and the diameter of $\Omega$.
This is due to the reverse H\"older's inequality in \cite{arXiv:1702.07045v1}, the statement of which is of global nature for solutions with the Dirichlet boundary condition.
\end{remark}

\begin{proof}[Proof of Proposition \ref{170509_prop1}]
Without loss of generality, we assume that $x_0=0$.
Let $\mu,\, \nu$ be constants satisfying $1/\mu+1/\nu=1$ and $2\mu=q_0$, where $q_0\in (2,\infty)$ is a number from Lemma \ref{170506@lem1} that depends only on $d$, $\delta$, and $q$.

{\em{Case 1}}. $B_R\subset \Omega$.
By Assumption \ref{170219@ass2} $(\gamma)$ $(i)$, there exists a coordinate system such that
\begin{equation}		\label{170510@eq1}
\dashint_{B_R}\big|A^{\alpha\beta}(x_1,x')-A_0^{\alpha\beta}(x_1)\big|\,dx\le \gamma,
\end{equation}
where we set
$$
A_0^{\alpha\beta}(x_1)=\dashint_{B_R'}A^{\alpha\beta}(x_1,y')\,dy'.
$$
Let $\cL_0$ be the elliptic operator with the coefficients $A^{\alpha\beta}_0$ and let $\cB_0$ be the conormal derivative operator associated with $\cL_0$.

Note that in the ball $B_R$, the hypothesis of Lemma \ref{170224@lem2} is satisfied.
Hence, there exists a unique $(w,p_1)\in \tilde{W}^1_2(B_R)^d\times L_2(B_R)$ satisfying
$$
\left\{
\begin{aligned}
\operatorname{div} w=g &\quad \text{in }\, B_R,\\
\cL_0 w+\nabla p_1=D_\alpha F^\alpha &\quad \text{in }\, B_R,\\
\cB_0 w+np_1=F^\alpha n_\alpha &\quad \text{on }\, \partial B_R,
\end{aligned}
\right.
$$
where $F^\alpha=(A^{\alpha\beta}_0-A^{\alpha\beta})D_\beta u+f^\alpha$, and
\begin{equation}		\label{170510@eq1a}
\|Dw\|_{L_2(B_R)}+\|p_1\|_{L_2(B_R)}\le N\big(\|F^\alpha\|_{L_2(B_R)}+\|g\|_{L_2(B_R)}\big),
\end{equation}
where $N=N(d,\delta)$.
By H\"older's inequality, the boundedness of $A^{\alpha\beta}$, \eqref{170510@eq1}, and Lemma \ref{170506@lem1}, we have
\begin{align*}
(|F^\alpha|^2)^{\frac{1}{2}}_{B_R}&\le \big(|A^{\alpha\beta}_0-A^{\alpha\beta}|^{2\nu}\big)^{\frac{1}{2\nu}}_{B_R}(|Du|^{2\mu})^{\frac{1}{2\mu}}_{B_R}+(|f^\alpha|^2)^{\frac{1}{2}}_{B_R}\\
&\le N\gamma^{\frac{1}{2\nu}}(|Du|^2+|p|^2)^{\frac{1}{2}}_{B_{2R}}
+N(|f^\alpha|^{2\mu}+|g|^{2\mu})^{\frac{1}{2\mu}}_{B_{2R}},
\end{align*}
where $N=N(d,\delta,q)$.
Using this together with \eqref{170510@eq1a}, we obtain \eqref{170510@eq1b} with $W=Dw$.
To show \eqref{170510@eq1c}, we observe that $(v,p_2):=(u,p)-(w,p_1)$ satisfies
$$
\left\{
\begin{aligned}
\operatorname{div} v=0 &\quad \text{in }\, B_R,\\
\cL_0 v+\nabla p_2=0 &\quad \text{in }\, B_R.
\end{aligned}
\right.
$$
By Lemma \ref{160815@lem1} $(a)$ with scaling, we get
$$
\|Dv\|_{L_\infty(B_{R/2})}+\|p_2\|_{L_\infty(B_{R/2})}\le N(|Dv|^2+|p_2|^2)^{\frac{1}{2}}_{B_{2R}},
$$
where $N=N(d,\delta)$.
From this, the fact that $(v,p_2)=(u-w,p-p_1)$, and \eqref{170510@eq1b}, we conclude \eqref{170510@eq1c} with $V=Dv$.

{\em{Case 2}}. $0\in \partial \Omega$.
By Assumption \ref{170219@ass2} $(\gamma)$ $(ii)$, there is a coordinate system such that \eqref{170510@eq1} and
$$
\{y:\gamma R<y_1\}\cap B_R\subset \Omega_R\subset \{y:-\gamma R<y_1\}\cap B_R.
$$
Define $\cL_0$ and $\cB_0$ as in {\em{Case 1}}, and set
$$
\tilde{B}_{R}^+=B_{R}\cap \{y:\gamma R<y_1\}.
$$
We note that $(u,p)$ satisfies
\begin{equation}		\label{170511@eq1}
\left\{
\begin{aligned}
\operatorname{div} u=g
&\quad \text{in }\, \tilde{B}_R^+,\\
\cL_0 u+\nabla p=D_\alpha (F^\alpha+G^\alpha)
&\quad \text{in }\, \tilde{B}_R^+,\\
\cB_0 u+np=(F^\alpha+G^\alpha) n_\alpha
&\quad \text{on }\, B_R\cap \{y:\gamma R=y_1\},
\end{aligned}
\right.
\end{equation}
where we set
$$
F^\alpha=(A_0^{\alpha\beta}-A^{\alpha\beta})D_\beta u+f^\alpha,
$$
\begin{align*}
G^1(y_1,y')&=A^{1\beta}(2\gamma R-y_1,y')(D_{\beta}u)(2\gamma R-y_1,y')\chi_{\Omega^*}\\
&\quad -f^1(2\gamma R-y_1,y')\chi_{\Omega^*}+e_1 p(2\gamma R-y_1,y')\chi_{\Omega^*},
\end{align*}
and for $\alpha\in \{2,\ldots,d\}$,
\begin{align*}
G^\alpha(y_1,y')&=-A^{\alpha\beta}(2\gamma R-y_1,y')(D_{\beta}u)(2\gamma R-y_1,y')\chi_{ \Omega^*}\\
&\quad +f^\alpha(2\gamma R-y_1,y')\chi_{\Omega^*}-e_\alpha p(2\gamma R-y_1,y')\chi_{\Omega^*}.
\end{align*}
Here, $e_\alpha$ is the $\alpha$-th unit vector in $\bR^d$ and
$$
\Omega^*=\{(y_1,y'): (2\gamma R-y_1,y')\in \Omega_R\setminus \tilde{B}^+_R\}\cap \tilde{B}^+_R.
$$
Indeed, one can check that \eqref{170511@eq1} holds as follows.
Let $\phi\in W^1_2(\tilde{B}_R^+)$ which vanishes on $\partial B_R\cap \{y:\gamma R<y_1\}$.
We extend $\phi$ to $\{y:\gamma R<y_1\}$ by setting $\phi\equiv 0$ on $\{y:\gamma R<y_1\}\setminus B_R$.
Then set
$$
\tilde{\phi}=
\left\{
\begin{aligned}
\phi(y_1,y') &\quad \text{if }\, y_1>\gamma R,\\
\phi(2\gamma R-y_1,y')&\quad \text{otherwise}.
\end{aligned}
\right.
$$
It is easily seen that $\tilde{\phi}\in W^1_2(\Omega_R)$ and vanishes on $\Omega\cap \partial B_R$.
Since $(u,p)$ satisfies \eqref{170511@eq2}, we have
$$
\int_{\Omega_R}A^{\alpha\beta}D_\beta u\cdot D_\alpha \tilde{\phi}\,dx
+\int_{\Omega_R}p \operatorname{div}\tilde{\phi}\,dx
=\int_{\Omega_R}f^\alpha\cdot D_\alpha \tilde{\phi}\,dx.
$$
From this identity and the definition of $\tilde{\phi}$, it follows that
\begin{align*}
&\int_{\tilde{B}_R^+}A^{\alpha\beta}_0D_\beta u\cdot D_\alpha \phi\,dx+\int_{\tilde{B}_R^+}p \operatorname{div}\phi\,dx\\
&=\int_{\tilde{B}_R^+}(A^{\alpha\beta}_0-A^{\alpha\beta})D_\beta u\cdot D_\alpha \phi\,dx-\int_{\Omega_R\setminus \tilde{B}_R^+}A^{\alpha\beta}D_\beta u\cdot D_\alpha \tilde{\phi}\,dx\\
&\quad -\int_{\Omega_R\setminus \tilde{B}_R^+}p \operatorname{div}\tilde{\phi}\,dx+\int_{\tilde{B}_R^+}f^\alpha \cdot D_\alpha \phi\,dx+\int_{\Omega_R\setminus \tilde{B}_R^+}f^\alpha\cdot D_\alpha \tilde{\phi}\,dx\\
&=\int_{\tilde{B}_R^+}(F^\alpha+G^\alpha)\cdot D_\alpha \phi\,dx,
\end{align*}
which implies that \eqref{170511@eq1} holds.

Since the hypothesis of Lemma \ref{170224@lem2} holds on $\tilde{B}_{R}^+$, there exists a unique $(\hat{w},\hat{p}_1)\in \tilde{W}^1_2(\tilde{B}_{R}^+)^d\times L_2(\tilde{B}_{R}^+)$ satisfying
$$
\left\{
\begin{aligned}
\operatorname{div} \hat{w}=g &\quad \text{in }\, \tilde{B}_{R}^+,\\
\cL_0 \hat{w}+\nabla \hat{p}_1=D_\alpha (F^\alpha+G^\alpha) &\quad \text{in }\, \tilde{B}_{R}^+,\\
\cB_0 \hat{w}+n\hat{p}_1=(F^\alpha+G^\alpha) n_\alpha &\quad \text{on }\, \partial \tilde{B}_{R}^+
\end{aligned}
\right.
$$
and
$$
\|D\hat{w}\|_{L_2(\tilde{B}^+_R)}+\|\hat{p}_1\|_{L_2(\tilde{B}^+_R)}\le N\big(\|F^\alpha+G^\alpha\|_{L_2(\tilde{B}^+_R)}+\|g\|_{L_2(\tilde{B}_R)}\big),
$$
where $N=N(d,\delta)$.
Note that
$$
\|G^\alpha\|_{L_2(\tilde{B}^+_R)}\le N\||Du|+|p|\|_{L_2(\Omega_R\setminus \tilde{B}^+_R)}
+N\|f^\alpha\|_{L_2(\Omega_R\setminus \tilde{B}^+_R)}.
$$
By H\"older's inequality, Lemma \ref{170506@lem1}, and the fact that
$$
|\Omega_R\setminus \tilde{B}^+_R|\le N\gamma R^d, \quad |\Omega_{2R}|\ge NR^d,
$$
we have
$$
(|G^\alpha|^2)^{\frac{1}{2}}_{\tilde{B}^+_R}\le N\gamma^{\frac{1}{2\nu}}(|Du|^2+|p|^2)^{\frac{1}{2}}_{\Omega_{2R}}+N(|f^\alpha|^{2\mu}+|g|^{2\mu})^{\frac{1}{2\mu}}_{\Omega_{2R}},
$$
where $N=N(d,\delta,q)$.
Using this and following the same arguments in the proof of \eqref{170510@eq1b} in {\em{Case 1}}, one can easily show that
$$
(|D\hat{w}|^2+|\hat{p}_1|^2)^{\frac{1}{2}}_{\tilde{B}^+_R}\le N\gamma^{\frac{1}{2\nu}}(|Du|^2+|p|^2)^{\frac{1}{2}}_{\Omega_{2R}}+N(|f^\alpha|^{2\mu}+|g|^{2\mu})^{\frac{1}{2\mu}}_{\Omega_{2R}}.
$$
Hence, defining $(W,p_1)$ in $\Omega_R$ by
$$
(W,p_1)=(D\hat{w},\hat{p}_1)\chi_{\tilde{B}^+_R}+(Du, p)\chi_{\Omega_R\setminus \tilde{B}^+_R},
$$
 we see that \eqref{170510@eq1b} holds.
To prove \eqref{170510@eq1c}, set
$$
(v,p_2)=(u-\hat{w},p-\hat{p}_1)\chi_{\tilde{B}_R^+} \quad \text{and}\quad V=Dv\chi_{\tilde{B}^+_R}.
$$
Observe that $(v,p_2)$ satisfies
$$
\left\{
\begin{aligned}
\operatorname{div} v=0 &\quad \text{in }\, \tilde{B}_{R}^+,\\
\cL_0 v+\nabla p_2=0 &\quad \text{in }\, \tilde{B}_{R}^+,\\
\cB_0 v+n p_2=0 &\quad \text{on }\, \partial \tilde{B}_{R}^+.
\end{aligned}
\right.
$$
We write $y_0=(\gamma R,0,\ldots,0)\in \bR^d$.
We then have
\begin{align*}
\Omega_{R/4}&\subset B_{R/4}(y_0)\cap \{y_1>\gamma R\}\\
&\subset B_{R/2}(y_0)\cap \{y_1>\gamma R\}\subset \Omega_{R}.
\end{align*}
Therefore, by Lemma \ref{160815@lem1} $(b)$ with scaling, we get
$$
\|Dv\|_{L_\infty(\Omega_{R/4})}+\|p_2\|_{L_\infty(\Omega_{R/4})}\le (|V|^2)^{\frac{1}{2}}_{\Omega_R},
$$
which together with \eqref{170510@eq1b} gives \eqref{170510@eq1c}.
The proposition is proved.
\end{proof}

We denote the maximal function of $f$ defined on $\bR^d$ by
$$
\cM f(x)=\sup_{\substack{y\in \bR^d, \,r>0\\ x\in B_r(y)}}\dashint_{B_r(y)} |f(z)|\,dz.
$$
Throughout this paper, by $\cM (f)$ we mean $\cM(f I_\Omega)$ if $f$ is defined on $\Omega\subset \bR^d$.
For $q\in (2,\infty)$, we denote
\begin{align*}
\cE_1(s)&=\{x\in \Omega:(\cM(|Du|^2+|p|^2)(x))^{1/2}>s\},\\
\cE_2(s)&=\big\{x\in \Omega:(\cM(|Du|^2+|p|^2)(x))^{1/2}\\
&\quad+\gamma^{-1/(2\nu)}(\cM(|f^\alpha|^{2\mu}+|g|^{2\mu})(x))^{1/(2\mu)} >s\big\},
\end{align*}
where $\mu,\nu\in (1,\infty)$ are from Proposition \ref{170509_prop1} satisfying $2\mu<q$ and $1/\mu+1/\nu=1$.

The proof of Theorem \ref{MT1} relies on the following estimate of level sets.

\begin{lemma}		\label{170512@lem5}
Suppose that Assumption \ref{170219@ass2} $(\gamma)$ holds with $\gamma\in (0,1/48]$.
Let $q\in (2,\infty)$ and $(u,p)\in W^1_2(\Omega)^d\times L_2(\Omega)$ satisfy \eqref{170511@eq2} with $f^\alpha\in L_q(\Omega)^d$ and $g\in L_q(\Omega)$.
Then there exists a constant $\kappa=\kappa(d,\delta,q)>1$ such that the following holds: for $x_0\in \overline{\Omega}$ and $R\in (0,R_0]$, and $s>0$, if
\begin{equation}		\label{170512@eq2}
\gamma^{\frac{1}{\nu}}\le \frac{|\Omega_{R/64}(x_0)\cap \cE_1(\kappa s)|}{|\Omega_{R/64}(x_0)|},
\end{equation}
then we have
$$
\Omega_{R/64}(x_0)\subset \cE_2(s).
$$
\end{lemma}

\begin{proof}
By scaling and translating the coordinates, we may assume that $s=1$ and $x_0=0$.
We prove by contradiction.
Suppose that
$$
(\cM(|Du|^2+|p|^2)(x))^{1/2}+\gamma^{-1/(2\nu)}
(\cM(|f^\alpha|^{2\mu}+|g|^{2\mu})(x))^{1/(2\mu)}\le 1
$$
for some $x\in \Omega_{R/64}$.
If $\operatorname{dist}(0,\partial \Omega)\ge R/16$, we have
$$
x\in B_{R/64}\subset B_{R/16}\subset \Omega.
$$
By Proposition \ref{170509_prop1}, $(Du,p)$ admits a decomposition
$$
(Du,p)=(W,p_1)+(V,p_2) \quad \text{in }\, B_{R/16}
$$
with the estimates
$$
(|W|^2+|p_1|^2)^{\frac{1}{2}}_{B_{R/16}}\le N_1\gamma^{\frac{1}{2\nu}}, \quad \|V\|_{L_\infty(B_{R/64})}+\|p_2\|_{L_\infty(B_{R/64})}\le N_1,
$$
where $N_1=N_1(d,\delta,q)$.
From this together with Chebyshev's inequality, it follows that
\begin{align*}
|B_{R/64}\cap \cE_1(\kappa)|&\le \big|\{ y\in B_{R/64}:|W(y)|+|p_1(y)|>\kappa -N_1\}\big|\\
&\le \int_{B_{R/64}}\left|\frac{|W|+|p|}{\kappa-N_1}\right|^2\,dy\le \frac{N_1^2\gamma^{\frac{1}{\nu}}}{|\kappa -N_1|^2}|B_{R/8}|,
\end{align*}
which contradicts \eqref{170512@eq2} if we choose a sufficiently large $\kappa$.

On the other hand, if $\operatorname{dist}(0,\partial \Omega)<R/16$, we take $z_0\in \partial \Omega$ such that $\operatorname{dist}(0,\partial \Omega)=|z_0|$.
Note that
$$
x\in \Omega_{R/64}\subset \Omega_{R/8}(z_0).
$$
By Proposition \ref{170509_prop1},
$(Du,p)$ admits a decomposition
$$
(Du,p)=(W,p_1)+(V,p_2) \quad \text{in }\, \Omega_{R/2}(z_0)
$$
with the estimates
$$
(|W|^2+|p_1|^2)^{\frac{1}{2}}_{\Omega_{R/2}}\le N_2\gamma^{\frac{1}{2\nu}}, \quad \|V\|_{L_\infty(\Omega_{R/8})}+\|p_2\|_{L_\infty(\Omega_{R/8})}\le N_2,
$$
where $N_2=N_2(d,\delta,q)$.
We then obtain that
$$
|\Omega_{R/16}\cap \cE_1(\kappa)|\le \frac{N_2^2\gamma^{\frac{1}{\nu}}}{|\kappa-N_2|^2}|B_R|,
$$
which contradicts \eqref{170512@eq2} if we choose a sufficiently large $\kappa$.
\end{proof}

\begin{proof}[Proof of Theorem \ref{MT1}]
We may assume that $f\equiv 0$.
Indeed, by the result in \cite{MR2263708} and the fact that $\Omega$ is a John domain (see, for instance, \cite[Remark 3.3]{ arXiv:1702.07045v1}), there exists $\phi^i\in \mathring{W}^1_{q_1}(\Omega)^d$ satisfying
$$
\operatorname{div} \phi^i=f_i \quad \text{in }\Omega, \quad \|D\phi^i\|_{L_{q_1}(\Omega)}\le N\|f_i\|_{L_{q_1}(\Omega)},
$$
where $N=N(d,q_1, R_0, K)$.
If we set $\Phi^\alpha=(\phi^1_\alpha,\ldots,\phi^d_\alpha)$, then by the Poincar\'e inequality we have
$$
\sum_{\alpha=1}^d D_\alpha \Phi^\alpha=f \quad \text{and}\quad \|\Phi^\alpha\|_{L_{q}(\Omega)}
\le N\|D\Phi^\alpha\|_{L_{q_1}(\Omega)}\le N\|f\|_{L_{q_1}(\Omega)}.
$$
In addition, because of Lemma \ref{170224@lem2}, we only need to consider the case when $q\neq 2$.

{\em{Case 1}}. $q>2$.
Let $\gamma\in (0,1/48]$ be a constant to be chosen below.
By Assumption \ref{170219@ass2} $(\gamma)$, $\Omega$ satisfies the hypothesis of Lemma \ref{170224@lem2}.
Hence, there exists a unique $(u,p)\in \tilde{W}^1_2(\Omega)^d\times L_2(\Omega)$ satisfying \eqref{170512@eq4}.
We prove that this $(u,p)$ is indeed in $W^1_q(\Omega)^d\times L_q(\Omega)$ and satisfies \eqref{170512@eq4a}.

Let $\kappa=\kappa(d,\delta,q)>1$ be the constant in Lemma \ref{170512@lem5}.
By the Hardy-Littlewood inequality, we have
\begin{equation}	\label{170512@eq8a}	
|\cE_1(\kappa s)| \le \frac{N_0}{(\kappa s)^{2}}\big\||Du|^2+|p|^2\big\|_{L_{1}(\Omega)}\le \frac{N_0}{(\kappa s)^{2}}\big\||Du|+|p|\big\|_{L_{2}(\Omega)}^{2},
\end{equation}
where $N_0=N_0(d,\delta)$.
Using this, Lemma \ref{170512@lem5}, and a result from measure theory on the ``crawling of ink spots,'' which can be found in \cite{MR579490} or \cite[Section 2]{MR563790}, we have
\begin{equation}		\label{170512@eq8}
|\cE_1(\kappa s)|\le N(d)\gamma^{1/\nu}|\cE_2(s)|
\end{equation}
for any $s>s_0$, where
$$
s_0^{2}= \frac{N_0}{\kappa^{2} \gamma^{1/\nu}|B_{R_0}|}\big\||Du|+|p|\big\|^{2}_{L_{2}(\Omega)}.
$$
For any sufficiently large $S>0$, it follows from \eqref{170512@eq8a} and \eqref{170512@eq8} that
\begin{align}
\nonumber
&\int_0^{\kappa S}|\cE_1(s)|s^{q-1}\,ds=\kappa^{q}\int^{S}_0|\cE_1(\kappa s)|s^{q-1}\,ds\\
\nonumber
&\le \kappa^q\int_0^{s_0}|\cE_1(\kappa s)|s^{q-1}\,ds+N_2\gamma^{1/\nu}\int_0^{S}|\cE_2(s)|s^{q-1}\,ds\\
\label{170512@eq9}
&\le N_1\gamma^{(2-q)/(2\nu)}\big\||Du|+|p|\big\|_{L_{2}(\Omega)}^q+N_2\gamma^{1/\nu}\int_0^{S}|\cE_2(s)|s^{q-1}\,ds,
\end{align}
where $N_1=N_1(d,\delta,q,R_0)$ and $N_2=N_2(d,\delta,q)$.
Using the definitions of $\cE_1$ and $\cE_2$, and the Hardy-Littlewood maximal function theorem, it holds that (use $q>2\mu$ and $\kappa>1$)
\begin{align}
\nonumber
&\gamma^{1/\nu}\int_0^{S}|\cE_2(s)|s^{q-1}\,ds\\
\nonumber
&\le N \gamma^{1/\nu}\int_0^{S}|\cE_1(s)|s^{q-1}\,ds+N\gamma^{(2-q) /(2\nu)}\||f^\alpha|+|g|\|_{L_q(\Omega)}^q\\
\label{170613@eq8}
&\le N \gamma^{1/\nu}\int_0^{\kappa S}|\cE_1(s)|s^{q-1}\,ds+N\gamma^{(2-q) /(2\nu)}\||f^\alpha|+|g|\|_{L_q(\Omega)}^q.
\end{align}
Notice from $\operatorname{diam}\Omega\le K<\infty$ that
$$
\int_0^{\kappa S}|\cE_1(s)|s^{q-1}\,dx<\infty.
$$
Then, by taking $\gamma=\gamma(d,\delta,q)$ sufficiently small in \eqref{170613@eq8}, we obtain from \eqref{170512@eq9} that
$$
\int_0^{\kappa S}|\cE_1(s)|s^{q-1}\,ds\le N\left(\|Du\|_{L_{2}(\Omega)}^q+\|p\|_{L_{2}(\Omega)}^q+\|f^\alpha\|_{L_q(\Omega)}^q+\|g\|_{L_q(\Omega)}^q\right),
$$
where $N=N(d,\delta,q,R_0)$.
Now, let $S\to \infty$, and use Lemma \ref{170224@lem2} and H\"older's inequality to obtain
$$
\|\cM(|Du|^2+|p|^2)\|_{L_{q/2}(\Omega)}^{q/2}\le N\left(\|f^\alpha\|_{L_q(\Omega)}^q+\|g\|_{L_q(\Omega)}^q\right),
$$
where $N=N(d,\delta,q,R_0,K)$.
From this, we finally see that $(u,p)$ is in $W^1_q(\Omega)^d\times L_q(\Omega)$ and satisfies \eqref{170512@eq4a} because by the Lebesgue differentiation theorem
$$
|Du|^2+|p|^2\le \cM(|Du|^2+|p|^2)
$$
for a.e. $x\in \Omega$.

{\em{Case 2}}. $q\in (1,2)$.
We first prove the a priori estimate \eqref{170512@eq4a} by using a duality argument.
Suppose that $(u,p)\in \tilde{W}^1_q(\Omega)^d\times L_q(\Omega)$ satisfies \eqref{170512@eq4} with $f\equiv 0$.
Let $q'=q/(q-1)>2$ and $\gamma=\gamma(d,\delta,q')$ from {\em{Case 1}}.
Then, for any $F^\alpha\in L_{q'}(\Omega)^d$ and $G\in L_{q'}(\Omega)$, there exists a unique $(v,\pi)\in \tilde{W}^1_{q'}(\Omega)^d\times L_{q'}(\Omega)$ satisfying
\begin{equation}		\label{170513_eq5}
\left\{
\begin{aligned}
\operatorname{div} v=G \quad &\text{in }\ \Omega,\\
\cL^* v+\nabla \pi=D_\alpha  F^\alpha \quad &\text{in }\ \Omega,\\
\cB^* v+n \pi = F^\alpha n_\alpha \quad &\text{on }\ \partial \Omega,
\end{aligned}
\right.
\end{equation}
where $\cL^*$ is the adjoint operator of $\cL$, i.e.,
$$
\cL^* v=D_\alpha\big((A^{\beta\alpha})^TD_\beta v\big), \quad (A^{\beta\alpha})^T_{ij}=A^{\beta\alpha}_{ji},
$$
and $\cB^* v$ is the conormal derivative of $v$ associated with $\cL^*$.
We also have
\begin{equation}		\label{170513@eq1}
\|Dv\|_{L_{q'}(\Omega)}+\|\pi\|_{L_{q'}(\Omega)}\le N\big(\|F^\alpha\|_{L_{q'}(\Omega)}+\|G\|_{L_{q'}(\Omega)}\big),
\end{equation}
where $N=N(d,\delta,q,R_0,K)$.
We test \eqref{170513_eq5} by $u$ to obtain
$$
\int_\Omega A^{\alpha\beta}D_\beta u\cdot D_\alpha v\,dx+\int_\Omega \pi \operatorname{div}u\,dx=\int_\Omega F^\alpha\cdot D_\alpha u\,dx.
$$
Since $(u,p)$ satisfies \eqref{170512@eq4}, it follows from the above identity that
$$
\int_\Omega F^\alpha \cdot D_\alpha u\,dx+\int_\Omega p G\,dx=\int_\Omega f^\alpha \cdot D_\alpha v\,dx+\int_\Omega \pi g\,dx.
$$
Using this and \eqref{170513@eq1}, we get
\begin{align*}
&\left|\int_\Omega F^\alpha \cdot D_\alpha u\,dx+\int_\Omega p G\,dx\right|\\
&\le N\big(\|F^\alpha\|_{L_{q'}(\Omega)}+\|G\|_{L_{q'}(\Omega)}\big)\big(\|f^\alpha\|_{L_{q}(\Omega)}+\|g\|_{L_{q}(\Omega)}\big)
\end{align*}
for any $F^\alpha\in L_{q'}(\Omega)^d$ and $G\in L_{q'}(\Omega)$.
Thus we have  the estimate \eqref{170512@eq4a}.

Next, we prove the solvability.
For $k>0$, we set
$$
f^{\alpha,k}=\max\{-k, \min\{f^\alpha,k\}\}, \quad g^k=\max\{-k, \min\{g,k\}\}.
$$
Since $f^{\alpha,k}$ and $g^k$ are bounded, by Lemma \ref{170224@lem2}, there exists a unique  $(u^k,p^k)\in \tilde{W}^1_2(\Omega)^d\times L_2(\Omega)$ satisfying
\begin{equation}		\label{170513_eq6}
\left\{
\begin{aligned}
\operatorname{div} u^k=g^k \quad &\text{in }\ \Omega,\\
\cL u^k+\nabla p^k=D_\alpha  f^{\alpha,k} \quad &\text{in }\ \Omega,\\
\cB u^k+np^k= f^{\alpha,k} n_\alpha \quad &\text{on }\ \partial \Omega.
\end{aligned}
\right.
\end{equation}
Note that $(u^k, p^k)\in \tilde{W}^1_q(\Omega)^d\times L_q(\Omega)$.
Hence, by the a priori estimate, we have
\begin{align*}
\|Du^k\|_{L_q(\Omega)}+\|p^k\|_{L_q(\Omega)}&\le N\big(\|f^{\alpha,k}\|_{L_q(\Omega)}+\|g^k\|_{L_q(\Omega)}\big)\\
&\le N\big(\|f^{\alpha}\|_{L_q(\Omega)}+\|g\|_{L_q(\Omega)}\big),
\end{align*}
where $N=N(d,\delta,q,R_0,K)$.
From the weak compactness, there exist a subsequence $(u^{k_j}, p^{k_j})$ and $(u, p)\in \tilde{W}^1_q(\Omega)\times L_q(\Omega)$ such that
$u^{k_j} \rightharpoonup u$ in $ W^1_q(\Omega)$ and $p^{k_j}\rightharpoonup p$ in  $L_q(\Omega)$.
From \eqref{170513_eq6}, it is routine to check that $(u,p)$ satisfies \eqref{170512@eq4} with $f\equiv 0$.
Finally, the uniqueness is a simple consequence of \eqref{170512@eq4a}.
The theorem is proved.
\end{proof}

\section{Appendix: Poincar\'e inequality}		\label{S_A}

In this section, we provide a detailed proof of a local version of the Poincar\'e inequality on a Reifenberg flat domain.
Throughout the appendix, we denote the line segment connecting $x$ and $y$ by $\overline{xy}$.

\begin{lemma}		\label{170425@lem2}
Let $\Omega$ be a Reifenberg flat domain satisfying Assumption \ref{170226@ass1} with $\gamma\in [0,1/48]$.
Let $x_0\in \partial \Omega$ and $R\in (0,R_0]$.
Then there exists a sequence $\{x^k\}_{k=0}^\infty$ such that $x^k\in \partial B_{R/2^{k+1}}(x_0)\cap \Omega$,
$$
\overline{x^k\,  x^{k+1}}\subset \overline{B_{R/2^{k+1}}(x_0)\setminus B_{R/2^{k+2}}(x_0)}\cap \Omega, \quad \ell(\overline{x^k\,  x^{k+1}}) \le \frac{R}{2^k},
$$
\begin{equation}
							\label{eq0428_01}
\operatorname{dist}(z,\partial \Omega)> 4 \gamma \frac{R}{2^k}
\end{equation}
for $z\in \overline{x^k\,x^{k+1}}$.
Moreover, for any $x \in \Omega$ satisfying
\begin{equation}
							\label{eq0428_02}
\operatorname{dist}(x,\partial\Omega) = |x-x_0| \le \frac{R}{4},
\end{equation}
there exists $k_0\in \{0,1,2,\ldots\}$ such that $x\in \overline{B_{R/2^{k_0+2}}(x_0)}\cap \Omega$,
$$
\overline{x x^{k_0}} \subset \overline{B_{R/2^{k_0+1}}(x_0)}\cap \Omega,
\quad \ell(\overline{x\, x^{k_0}}) \le \frac{R}{2^{k_0}},
$$
and \eqref{eq0428_01} is satisfied for all $z \in \overline{x x^{k_0}}$ with $k_0$ in place of $k$.
Here, $\ell(\overline{x\, x^{k_0}})$ is the length of $\overline{x\, x^{k_0}}$.
\end{lemma}

\begin{proof}
Set $\rho_k=R/2^k$.
By Assumption \ref{170226@ass1}, there exists a coordinate system associated with $(x_0,\rho_k)$ satisfying \eqref{170512_eq1} with $\rho_k$ in place of $R$.
Then, let $x^k$ be the intersection of the boundary of the ball $B_{\rho_k/2}(x_0)$ and the positive $y_1$-axis of the coordinate system associated with $(x_0,\rho_k)$.
Note that, as $k$ changes, the $y_1$-direction of the coordinate system may differ because, for each $k\in \{0,1,\ldots \}$, a coordinate system is chosen depending on $(x_0,\rho_k)$.
Since it holds that
$$
B_{(1/2-\gamma)\rho_k}(x^k) \subset \Omega_{\rho_k}(x_0),
$$
we have
\begin{equation}		\label{170426@eq2}
\operatorname{dist}(x^k,\partial \Omega)\ge (1/2-\gamma)\rho_k.
\end{equation}
By repeating the above argument we choose $x^{k+1}\in \Omega_{\rho_{k+1}}(x_0)$ satisfying
\begin{equation}		\label{170425@eq4}
\operatorname{dist}(x^{k+1},\partial \Omega)\ge (1/2-\gamma)\rho_{k+1} = (1/4-\gamma/2)\rho_k.
\end{equation}
In particular, $x^k$ and $x^{k+1}$ are located in the half space
\begin{equation}
							\label{eq0428_03}
H^+_k:=\left\{y:x_{01}+5 \gamma \rho_k<y_1\right\},
\end{equation}
where the coordinate system $(y_1,y')$ is that associated with $(x_0, \rho_k)$.
Indeed, it is clear that $x^k\in H^+_k$.
If $x^{k+1}\in \bR^d\setminus H^+_k$, then by Assumption \ref{170226@ass1}, it belongs to
$$
\left\{y:x_{01}-\gamma \rho_k<y_1\le x_{01}+5\gamma\rho_k\right\},
$$
which implies that
$$
\operatorname{dist}(x^{k+1},\partial \Omega)\le 6\gamma\rho_k.
$$
This contradicts \eqref{170425@eq4}.
Hence, we have
$$
x^k, \, x^{k+1}\in H^+_k\cap \overline{B_{\rho_k/2}(x_0)} \subset  \left\{y:x_{01}+\gamma \rho_k<y_1\right\} \cap B_{\rho_k}(x_0) \subset \Omega.
$$
This shows that
$$
\overline{x^k\, x^{k+1}}\subset \overline{B_{\rho_k/2}(x_0)\setminus B_{\rho_{k+1}/2}(x_0)}\cap \Omega, \quad \ell(\overline{x^k\, x^{k+1}}) \le \rho_k,
$$
and
$$
 \operatorname{dist}(z,\partial \Omega)> 4\gamma \rho_k \quad \text{for }\, z\in \overline{x^k\,x^{k+1}}.
$$

Now let $x\in \Omega$ satisfy \eqref{eq0428_02}.
Find an integer $k_0 \in \{0,1,2,\ldots\}$ satisfying
\begin{equation}
							\label{eq0430_02}
\frac{\rho_{k_0}}{8} < \operatorname{dist}(x,\partial\Omega) = |x-x_0| \le \frac{\rho_{k_0}}{4},
\end{equation}
so that $x \in \overline{B_{\rho_{k_0}/4}(x_0)}$.
We see that $x \in H_{k_0}^+$, where $H_{k_0}^+$ is as in \eqref{eq0428_03} with the coordinate system associated with $(x_0, \rho_{k_0})$.
Otherwise, we have $\operatorname{dist}(x,\partial\Omega) \le 6\gamma \rho_{k_0}$, which contradicts the first inequality in \eqref{eq0430_02}.
Hence, it follows that
$$
\overline{x\, x^{k_0}} \in H_{k_0}^+ \cap \overline{B_{\rho_{k_0}/2}(x_0)} \subset \left\{y: {x_0}_1 + \gamma \rho_{k_0} < y_1\right\} \cap B_{\rho_{k_0}}(x_0) \subset \Omega,
$$
$$
\ell(\overline{x\,  x^{k_0}}) \le \rho_{k_0},
$$
and
\eqref{eq0428_01} is satisfied for all $z \in \overline{x\, x^{k_0}}$ with $k_0$ in place of $k$.
The lemma is proved.
\end{proof}

\begin{proposition}		\label{MTP}
Let $\Omega$ be a Reifenberg flat domain satisfying Assumption \ref{170226@ass1} with $\gamma\in [0,1/48]$.
Let $0\in \partial \Omega$ and $R\in (0,R_0/4]$, and
$$
z_0:=(R/2,0,\ldots,0)
$$
in the coordinate system associated with the origin $0$ and $4R$ satisfying
\begin{equation}		\label{170426@eq1}
\{y:\gamma 4R<y_1\}\cap B_{4R} \subset \Omega_{4R}\subset \{y:-\gamma 4R<y_1\}\cap B_{4R}.
\end{equation}
Then, for any $x\in \Omega_R$, there exists a rectifiable curve $\eta$ joining $x$ and $z_0$ such that $\ell(\eta)\le 5R$, $\eta\subset \Omega_{7R/4}$, and
\begin{equation}		\label{170426@eq9b}
\operatorname{dist}(z,\partial \Omega)\ge 2\gamma|x-z|, \quad \ell(\eta;x,z)\le \frac{1}{\gamma}|x-z|  \quad \text{for all }\, z\in \eta.
\end{equation}
Here, $\ell(\eta)$ is the length of $\eta$, and $\ell(\eta;x,z)$ is the length of $\eta$ from $x$ to $z$.
\end{proposition}

\begin{remark}		\label{170427@rmk1}
Every rectifiable curve in metric space can be parametrized by arclength. %;  see \cite{MR0075623, MR0454009}.
Indeed, the rectifiable curve $\eta$ constructed in Proposition \ref{MTP} can be parametrized as follows.
Set
$$
s:=\frac{\ell(\eta;x,z)}{\ell(\eta)}\in [0,1]  \quad \text{and}\quad \eta(s)=z.
$$
Then $\eta=\eta(s)$ is a curve on $[0,1]$ satisfying
\begin{equation}		\label{170502@eq1}
\operatorname{Lip}(\eta)\le 5R, \quad \operatorname{dist}(\eta(s),\partial \Omega)> \frac{1}{4\cdot 24\cdot 48}Rs \quad \text{for }\, s\in [0,1].
\end{equation}
Indeed, for any $s_1,\,s_2\in [0,1]$ with
$$
s_1=\frac{\ell(\eta;x,z^1)}{\ell(\eta)}, \quad s_2=\frac{\ell(\eta;x,z^2)}{\ell(\eta)},
$$
we have
\begin{align*}
\frac{|\eta(s_1)-\eta(s_2)|}{|s_1-s_2|}&=|z^1-z^2|\frac{\ell(\eta)}{|\ell(\eta;x,z^1)-\ell(\eta;x,z^2)|}\\
&=|z^1-z^2|\frac{\ell(\eta)}{|\ell(\eta;z^1,z^2)|}\le \ell(\eta)\le 5R,
\end{align*}
which implies that $\operatorname{Lip}(\eta)\le 5R$.
To prove the second inequality in \eqref{170502@eq1}, we may assume that $\gamma=1/48$.
If $|\eta(s)-z_0|\le R/4$, then (use $s\le 1$)
$$
\operatorname{dist}(\eta(s), \partial \Omega)\ge \frac{R}{4}-\frac{4 R}{48}= \frac{R}{6}\ge \frac{Rs}{6}.
$$
On the other hand, if $|\eta(s)-z_0|>R/4$, then by \eqref{170426@eq9b} with $\gamma=1/48$, we have
$$
s=\frac{\ell(\eta;x,\eta(s))}{\ell(\eta)}\le \frac{48|x-\eta(s)|}{|\eta(s)-z_0|}<\frac{4\cdot 24 \cdot 48}{R}\operatorname{dist}(\eta(s),\partial \Omega).
$$
\end{remark}

\begin{proof}[Proof of Proposition \ref{MTP}]
In this proof, we fix the coordinate system associated with the origin $0$ and $4R$ satisfying \eqref{170426@eq1}.
Let $x\in \Omega_R$.
If $x$ is located in the half space
$$
\left\{y: 8\gamma R <y_1\right\},
$$
then one can easily check that $\eta=\overline{x\,  z_0} \subset \Omega_R$ satisfies the assertions in Proposition \ref{MTP}.
In particular,
$$
\operatorname{dist}(z, \partial\Omega) \ge 4 \gamma R \ge 2 \gamma |x-z|
$$
for $z \in \overline{x\,z_0}$.

Now suppose that
\begin{equation}		\label{170427@eq1}
x\in\left\{y:y_1\le 8\gamma R\right\},
\end{equation}
which implies that $\operatorname{dist}(x,\partial \Omega)\le 8\gamma R + 4 \gamma R \le R/4$ (recall that $\gamma \le 1/48$).
Let $x_0$  be a point on $\partial \Omega$ such that
$$
\operatorname{dist}(x,\partial \Omega)=|x-x_0|\le \frac{R}{4},
$$
and observe that
$$
|x_0|\le |x_0-x|+|x|<\frac{5R}{4}.
$$
By applying Lemma \ref{170425@lem2} to $\Omega_R(x_0)$ we obtain
$$
\{x^0, x^1, \ldots,x^{k_0}\} \subset \Omega_R(x_0)
$$
such that
\begin{equation}
							\label{170501@eq3}
x \in \overline{B_{R/2^{k_0+2}}(x_0)} \cap \Omega, \quad \overline{x\,x^{k_0}}\subset \overline{B_{R/2^{k_0+1}}(x_0)}\cap \Omega,
\end{equation}
\begin{equation}		\label{170501@eq3a}
\ell(\overline{x\, x^{k_0}}) \le \frac{R}{2^{k_0}}, \quad \operatorname{dist}(z,\partial\Omega) > 4 \gamma \frac{R}{2^{k_0}} \quad \text{for} \,\, z\in \overline{x\, x^{k_0}},
\end{equation}
and, in case $k_0 \ge 1$,
\begin{equation}		\label{170501@eq3b}
\overline{x^k \, x^{k+1}} \subset \overline{B_{R/2^{k+1}}(x_0) \setminus B_{R/2^{k+2}}(x_0)} \cap \Omega,
\end{equation}
\begin{equation}		
							\label{170501@eq1}
\ell(\overline{x^k\, x^{k+1}}) \le \frac{R}{2^k}, \quad  \quad \operatorname{dist}(z,\partial\Omega) > 4 \gamma \frac{R}{2^k} \quad \text{for} \,\, z \in \overline{x^k \, x^{k+1}},
\end{equation}
where $k\in \{0,\ldots,k_0-1\}$.
To construct a curve joining $x$ and $z_0$, we connect $x$ to $z_0$ by the line segment $\overline{x \, x^{k_0}}$ and $\overline{x^k \, x^{k+1}}$, $k=0,1,\ldots,k_0$, up to $x^0$, and then connect $x^0$ to $z_0$ by the line segment $\overline{x^0 \, z_0}$.
Precisely, we connect $x$ and $z_0$ by the curve
$$
\eta=\eta_1\cup \overline{x^0\, z_0},
$$
where $\eta_1$ is a curve defined by
$$
\eta_1=
\left\{
\begin{aligned}
&\overline{x\,x^{k_0}} \quad \text{if} \quad k_0 = 0,
\\
&\overline{x\,x^{k_0}}\cup \left(\bigcup_{k=0}^{k_0-1}\overline{x^k\, x^{k+1}}\right) \quad \text{if} \quad k_0 \ge 1.
\end{aligned}
\right.
$$
Since $x_0\in B_{5R/4}$ and $|z-x_0|\le R/2$ for $z\in \eta_1$, we have
$$
\eta_1\subset \Omega_{7R/4}.
$$
We also obtain by \eqref{170501@eq3a} and \eqref{170501@eq1} that
$$
\ell(\eta_1)\le 2R.
$$
Note that the line segment $\overline{x^0 \, z_0}$ satisfies
\begin{equation}
							\label{eq0429_02}
\overline{x^0\,  z_0} \subset \Omega_{7R/4}, \quad \ell(\overline{x^0\,  z_0}) \le 3R, \quad \operatorname{dist}(z, \partial \Omega) > 12\gamma R
\end{equation}
for $z \in \overline{x^0 \, z_0}$.
Indeed, since $x_0 \in B_{5R/4}$ and $x^0 \in \partial B_{R/2}(x_0)$, it follows that $x^0 \in B_{7R/4}$.
This together with the fact that $z_0 \in \overline{B_{R/2}}$ shows that
\begin{equation}
							\label{eq0430_03}
\overline{x^0 \, z_0} \subset B_{7R/4}, \quad
\ell(\overline{x^0 \, z_0}) \le 3R.
\end{equation}
Moreover, by the choice of $z^0$ and \eqref{170426@eq2} with $k=0$, we have
$$
\operatorname{dist}(z_0, \partial\Omega) > (1/2 - 4\gamma)R, \quad
\operatorname{dist}(x^0,\partial\Omega) > (1/2-\gamma) R.
$$
Using this together with the fact that $x^0, z_0 \in \Omega_{2R}$, we have
\begin{equation}
							\label{eq0430_01}
\overline{x^0\, z_0}\subset \{y:  16 \gamma R < y_1\} \cap B_{4R} \subset \{y: 4 \gamma R < y_1\} \cap B_{4R} \subset \Omega.
\end{equation}
This along with \eqref{eq0430_03} proves \eqref{eq0429_02}.

From \eqref{eq0429_02} and the fact that $\eta_1\subset \Omega_{7R/4}$ and $\ell(\eta_1)\le 2R$, it follows that
$$
\eta\subset \Omega_{7R/4} \quad \text{and}\quad \ell(\eta)\le 5R.
$$

We next prove the first inequality in \eqref{170426@eq9b}.
For $z\in \overline{x\, x^{k_0}}$, by \eqref{170501@eq3a} we see that
$$
\operatorname{dist}(z,\partial \Omega)> 4 \gamma \frac{R}{2^{k_0}}\ge 4 \gamma |x-z|.
$$
If $z\in \overline{x^k\, x^{k+1}}$ with $k\in \{0,\ldots,k_0-1\}$ and $k_0 \ge 1$, then by \eqref{170501@eq3b}, \eqref{170501@eq1}, and the fact that $x \in \overline{B_{R/2^{k_0+2}}(x_0)}$, we have
$$
|x-z|\le |x-x_0|+|x_0-z|\le \frac{R}{2^{k_0+2}}+\frac{R}{2^{k+1}} \le \frac{R}{2^k} <\frac{1}{4 \gamma}\operatorname{dist}(z,\partial \Omega).
$$
For the case $z \in \overline{x^0\, z_0}$, since $z \in \Omega_{2R}$ and $x \in \Omega_R$, we have
$$
|x-z| < 3R,
$$
which, when combined with the last inequality in \eqref{eq0429_02}, proves that
$$
4\gamma |x-z| < \operatorname{dist}(z,\partial\Omega)
$$
for $z \in \overline{x^0\, z_0}$.

To prove the second inequality in \eqref{170426@eq9b}, we first observe that
$\ell(\eta;x,z)=|x-z|$ for $z\in \overline{x \, x^{k_0}}$.
If $z\in \overline{x^k\, x^{k+1}}$ with $k\in \{0,\ldots,k_0-1\}$ and $k_0 \ge 1$, from \eqref{170501@eq3} and \eqref{170501@eq3b} we have
$$
\quad |x-x_0|\le \frac{R}{2^{k_0+2}}  \quad \text{and} \quad \frac{R}{2^{k+2}} \le |z-x_0|\le \frac{R}{2^{k+1}},
$$
which shows that
\begin{align*}
|z-x|&\ge \frac{R}{2^{k+2}}-\frac{R}{2^{k_0+2}} \\
&= \frac{1}{4}\left( \frac{R}{2^{k+1}} + \cdots + \frac{R}{2^{k_0}} \right)\ge \frac{1}{24}\left( \frac{R}{2^{k}} + \cdots + \frac{R}{2^{k_0}} \right)\\
& \ge \frac{1}{24} \left(\ell(\overline{z \, x^{k+1}}) + \cdots +  \ell(\overline{x^{k_0} \, x^{k_0-1}}) + \ell(\overline{x \, x^{k_0}})\right)
= \frac{1}{24}\ell(\eta,x,z).
\end{align*}
For the case $z \in \overline{x^0\, z_0}$, from \eqref{170427@eq1}, \eqref{eq0430_01}, and the fact that $\ell(\eta)\le 5R$, it follows that
$$
|x-z| \ge 8 \gamma R \ge \gamma \ell(\eta) \ge  \gamma \ell(\eta;x,z).
$$
The proposition is proved.
\end{proof}

\begin{proof}[Proof of Theorem \ref{170503@thm1}]
Without loss of generality, we assume that $x_0=0$.
Set $h=4\cdot 24\cdot 48$ and
$$
z_0:=(R/2,0,\ldots,0)
$$
in the coordinate system associated with the origin $0$ and $4R$ satisfying \eqref{170426@eq1}.
For $x\in \Omega_R$ and $z\in B_{R/h}(z_0)$, we define
$$
\tilde{\eta}(s;x,z)=\eta(s;x)+s(z-z_0), \quad 0\le s\le 1,
$$
where $\eta(s;x)(=\eta(s))$ is the parametrized curve joining $x$ and $z_0$ constructed in Proposition \ref{MTP} and Remark \ref{170427@rmk1}.
It then holds that
$$
\tilde{\eta}(0;x,z)=x, \quad \tilde{\eta}(1;x,z)=z, \quad \tilde{\eta}\subset \Omega, \quad
\dot{\tilde{\eta}}(s)=\dot{\eta}(s)+z-z_0.
$$
In particular, from the fact that
$$
\eta\subset \Omega_{7R/4} \quad \text{and}\quad z\in B_{R/h}(z_0)\subset B_{R/4}(z_0),
$$
we have
$\tilde{\eta}\subset \Omega_{2R}$.
Denote
$$
\bar{u}=\frac{1}{\|\phi\|_{L_1(\bR^d)}}\int_{\bR^d}u(z)\phi(z-z_0)\,dz = \frac{1}{\|\phi\|_{L_1(\bR^d)}}\int_{B_{R/h}(z_0)}u(z)\phi(z-z_0)\,dz,
$$
where $\phi$ is a smooth function in $\bR^d$ satisfying
$$
0\le \phi\le 1, \quad \phi\equiv 1 \, \text{ on }\, B_{R/2h}, \quad \operatorname{supp}\phi\subset B_{R/h}.
$$
We then have
\begin{align*}
u(x)-\bar{u}&=\frac{1}{\|\phi\|_{L_1(\bR^d)}}\int_{B_{R/h}(z_0)}\big(u(\tilde{\eta}(0;x,z))-u(\tilde{\eta}(1;x,z))\big)\phi(z-z_0)\,dz\\
&=\frac{-1}{\|\phi\|_{L_1(\bR^d)}}\int_{B_{R/h}(z_0)}\int_0^1 \dot{\tilde{\eta}}(s;x,z)\cdot \nabla u(\tilde{\eta}(s;x,z))\phi (z-z_0)\,ds\,dz.
\end{align*}
By setting $y=\tilde{\eta}(s;x,z)$ in the above identity, and using the fact that $y\in \Omega_{2R}$ and
$\|\phi\|_{L_1(\bR^d)}\ge N(d,h)R^d$, we get
\begin{align}
\nonumber
&|u(x)-\bar u|\\
\nonumber
&\le \frac{N}{R^d}\int_{\Omega_{2R}}\int_0^1\left|\dot{\eta}(s;x)+\frac{y-\eta(s;x)}{s}\right|\phi\left(\frac{y-\eta(s;x)}{s}\right)\frac{1}{s^d}\,ds\,| \nabla u(y)|\,dy\\
\label{170422@eq1a}
&= \frac{N}{R^d}\int_{\Omega_{2R}}G(x,y)|\nabla u(y)|\,dy,
\end{align}
where
$$
G(x,y)=\int_0^1 \left|\dot{\eta}(s;x)+\frac{y-\eta(s;x)}{s}\right|\phi\left(\frac{y-\eta(s;x)}{s}\right)\frac{1}{s^d}\,ds.
$$
Observe that from Remark \ref{170427@rmk1}
$$
|x-\eta(s;x)|=|\eta(0;x)-\eta(s;x)|\le 5Rs.
$$
This implies  that, for $y\in \Omega_{2R}$ with $|x-y|\ge (1/h+5)Rs$, we have
$$
\phi\left(\frac{y-\eta(s;x)}{s}\right)=0,
$$
because in this case
$$
|y-\eta(s;x)|\ge |x-y|-|x-\eta(s;x)|\ge \frac{Rs}{h}.
$$
Hence, for $x \in \Omega_R$ and $y \in \Omega_{2R}$, we have
$$
G(x,y)=\int^1_{c_0|x-y|/R}\left|\dot{\eta}(s;x)+\frac{y-\eta(s;x)}{s}\right|\phi\left(\frac{y-\eta(s;x)}{s}\right)\frac{1}{s^d}\,ds,
$$
where $c_0=\left(\frac{1}{h}+5\right)^{-1}$, $c_0|x-y|/R < 1$, and
$$
\left|\dot{\eta}(s;x)+\frac{y-\eta(s;x)}{s}\right|\le 5R+\left|\frac{y-x}{s}\right|+\left|\frac{x-\eta(s;x)}{s}\right|\le 16R.
$$
It then follows that
\begin{equation}		\label{170503@eq1}
G(x,y)\le \int_{c_0|x-y|/R}^1\frac{16R}{s^d}\,ds\le N(d)\frac{R^d}{|x-y|^{d-1}}
\end{equation}
for $x\in \Omega_R$ and $y\in \Omega_{2R}$.
Combining \eqref{170422@eq1a} and \eqref{170503@eq1}, we obtain for $x\in \Omega_R$ that
$$
|u(x)-\bar{u}|\le N\int_{\Omega_{2 R}}\frac{|\nabla u(y)|}{|x-y|^{d-1}}\,dy
=N\int_{\bR^d}\frac{|\nabla u(y)|I_{\Omega_{2 R}}}{|x-y|^{d-1}}\,dy.
$$
Using the estimate for fractional integrations, we see that
$$
\|u-\bar u\|_{L_{q^*}(\Omega_R)}\le N\|\nabla u\|_{L_{q}(\Omega_{2R})},
$$
where $q^*=dq/(d-q)$ and $N=N(d,q)$.
Therefore, by using the fact that
$$
\|u-(u)_{\Omega_R}\|_{L_{q^*}(\Omega_R)}\le \|u-\bar{u}\|_{L_{q^*}(\Omega_R)},
$$
we obtain the desired estimate.
The theorem is proved.
\end{proof}

\bibliographystyle{plain}

\def\cprime{$'$}

\end{document}